\documentclass[a4paper,reqno,10.5pt, oneside]{amsart}

\usepackage{amssymb}
\usepackage{amstext}
\usepackage{amsmath}
\usepackage{amscd}
\usepackage{amsthm}
\usepackage{amsfonts}
\usepackage{enumerate}
\usepackage{graphicx}
\usepackage{latexsym}
\usepackage{mathrsfs}
\usepackage{mathtools}
\usepackage{here}

\usepackage{tikz}
\usepackage{tikz-cd}
\usetikzlibrary{arrows}
\usetikzlibrary{decorations.markings}
\usetikzlibrary{patterns,decorations.pathreplacing}

\makeatletter
\pgfdeclarepatternformonly[\LineSpace]{my north west lines}{\pgfqpoint{-1pt}{-1pt}}{\pgfqpoint{\LineSpace}{\LineSpace}}{\pgfqpoint{\LineSpace}{\LineSpace}}
{
    \pgfsetcolor{\tikz@pattern@color} 
    \pgfsetlinewidth{0.6pt}
    \pgfpathmoveto{\pgfqpoint{0pt}{\LineSpace}}
    \pgfpathlineto{\pgfqpoint{\LineSpace + 0.1pt}{- 0.1pt}}
    \pgfusepath{stroke}
}
\makeatother 
\newdimen\LineSpace
\tikzset{
    line space/.code={\LineSpace=#1},
    line space=8.5pt
}

\usepackage{hyperref}

\newtheorem{theorem}{Theorem}[section]

\newtheorem{corollary}[theorem]{Corollary}
\newtheorem{lemma}[theorem]{Lemma}
\newtheorem{proposition}[theorem]{Proposition}
\newtheorem{definition-proposition}[theorem]{Definition-Proposition}

\theoremstyle{definition}
\newtheorem{definition}[theorem]{Definition}
\newtheorem{example}[theorem]{Example}

\theoremstyle{remark}
\newtheorem{remark}[theorem]{Remark}

\makeatletter
  
  \@addtoreset{equation}{section}
\makeatother

\newcommand{\calA}{\mathcal{A}}

\newcommand{\calC}{\mathcal{C}}
\newcommand{\calD}{\mathcal{D}}

\newcommand{\calL}{\mathcal{L}}
\newcommand{\calM}{\mathcal{M}}
\newcommand{\calO}{\mathcal{O}}

\newcommand{\rmV}{\mathrm{V}}
\newcommand{\rmX}{\mathrm{X}}
\newcommand{\rmY}{\mathrm{Y}}

\newcommand{\fkp}{\mathfrak{p}}

\newcommand{\ZZ}{\mathbb{Z}}
\newcommand{\QQ}{\mathbb{Q}}
\newcommand{\RR}{\mathbb{R}}

\newcommand{\TT}{\mathbb{T}}
\newcommand{\kk}{\Bbbk}

\newcommand{\bfa}{\mathbf{a}}
\newcommand{\bfe}{\mathbf{e}}
\newcommand{\bfu}{\mathbf{u}}
\newcommand{\bfx}{\mathbf{x}}
\newcommand{\sfM}{\mathsf{M}}
\newcommand{\sfN}{\mathsf{N}}

\newcommand{\Spec}{\operatorname{Spec}}

\newcommand{\Hom}{\operatorname{Hom}}

\newcommand{\Ker}{\operatorname{Ker}}

\newcommand{\Cl}{\operatorname{Cl}}
\newcommand{\End}{\operatorname{End}}
\newcommand{\gldim}{\mathrm{gl.dim \,}}
\newcommand{\pdim}{\mathrm{proj.dim}}

\newcommand{\add}{\mathsf{add}}

\newcommand{\mc}{\mathsf{mod}}
\newcommand{\refl}{\mathsf{ref}}

\begin{document}

\title[Conic divisorial ideals of Hibi rings and their applications to NCCRs]{Conic divisorial ideals of Hibi rings and their applications to non-commutative crepant resolutions}
\author[A Higashitani \and Y. Nakajima]{Akihiro Higashitani \and Yusuke Nakajima} 

\address[A Higashitani]{ Department of Pure and Applied Mathematics, Graduate School of Information Science and Technology, Osaka University, Suita, Osaka 565-0871, Japan}
\email{higashitani@ist.osaka-u.ac.jp}

\address[Y. Nakajima]{Kavli Institute for the Physics and Mathematics of the Universe (WPI), UTIAS, The University of Tokyo, Kashiwa, Chiba 277-8583, Japan} 
\email{yusuke.nakajima@ipmu.jp}



\maketitle

\begin{abstract} 
In this paper, we study divisorial ideals of a Hibi ring which is a toric ring arising from a partially ordered set.  
We especially characterize the special class of divisorial ideals called conic using the associated partially ordered set. 
Using our description of conic divisorial ideals, we also construct a module giving a non-commutative crepant resolution (= NCCR) 
of the Segre product of polynomial rings. 
Furthermore, applying the operation called mutation, we give other modules giving NCCRs of it.
\end{abstract}


\section{Introduction} 
\label{sec_intro}
In this paper, we study a certain class of toric rings called Hibi rings. Thus, we start this paper with introducing toric rings. 
Let $\sfN\cong\ZZ^d$ be a lattice of rank $d$ and let $\sfM\coloneqq\Hom_\ZZ(\sfN, \ZZ)$ be the dual lattice of $\sfN$. 
We set $\sfN_\RR\coloneqq\sfN\otimes_\ZZ\RR$ and $\sfM_\RR\coloneqq\sfM\otimes_\ZZ\RR$ and 
denote an inner product by $\langle\;,\;\rangle:\sfM_\RR\times\sfN_\RR\rightarrow\RR$. 
We consider a strongly convex rational polyhedral cone 
$$
\tau\coloneqq\mathrm{Cone}(v_1, \cdots, v_n)=\RR_{\ge 0}v_1+\cdots +\RR_{\ge 0}v_n\subset\sfN_\RR 
$$
of dimension $d$ generated by $v_1, \cdots, v_n\in\ZZ^d$ where $d\le n$. We assume this system of generators is minimal. 
For each generator, we define a linear form $\sigma_i(-)\coloneqq\langle-, v_i\rangle$ and denote $\sigma(-)\coloneqq(\sigma_1(-),\cdots,\sigma_n(-))$. 
We consider the dual cone $\tau^\vee$: 
$$
\tau^\vee\coloneqq\{{\bf x}\in\sfM_\RR \mid \sigma_i({\bf x})\ge0 \text{ for all } i=1,\cdots,n \}. 
$$
Then, $\tau^\vee\cap\sfM$ is a positive normal affine monoid, and hence we define the toric ring 
$$
R\coloneqq \kk[\tau^\vee\cap\sfM]=\kk[t_1^{m_1}\cdots t_d^{m_d}\mid (m_1, \cdots, m_d)\in\tau^\vee\cap\sfM], 
$$
where $\kk$ is an algebraically closed field. 
It is known that $R$ is a $d$-dimensional Cohen-Macaulay (= CM) normal domain. 
In addition, for each $\bfa=(a_1, \cdots, a_n)\in\RR^n$, we set 
$$
\TT(\bfa)\coloneqq\{{\bf x}\in\sfM \mid (\sigma_1({\bf x}), \cdots, \sigma_n({\bf x}))\ge(a_1, \cdots, a_n)\}. 
$$
Then, we define the module $T(\bfa)$ generated by all monomials whose exponent vector is in $\mathbb{T}(\bfa)$. 
By definition, we have $T(\bfa)=T(\ulcorner \bfa\urcorner)$, where $\ulcorner \; \urcorner$ means the round up 
and $\ulcorner \bfa\urcorner=(\ulcorner a_1\urcorner, \cdots, \ulcorner a_n\urcorner)$. 
This $T(\bfa)$ is a divisorial ideal (rank one reflexive module), and any divisorial ideal of $R$ takes this form (see e.g., \cite[Theorem~4.54]{BG2}). 
Therefore, each divisorial ideal is represented by $\bfa\in\ZZ^n$. Clearly, we have $\mathbb{T}(0)=\tau^\vee\cap\sfM$ and $T(0)=R$. 
In particular, we are interested in a divisorial ideal that is \emph{maximal Cohen-Macaulay} (= \emph{MCM}) as an $R$-module, 
that is, an $R$-module whose depth coincides with $\dim R=d$. 
Rank one MCM modules of toric rings have been studied in several papers e.g., \cite{Sta79,VdB1,Don,BG1,Bae, Bru}, 
and the number of isomorphism classes of rank one MCM modules is finite in particular \cite[Corollary~5.2]{BG1}. 
In what follows, we will pay attention to a certain class of divisorial ideals called \emph{conic class}. 
As we will see below, a conic divisorial ideal is a rank one MCM module, and hence the number of conic classes is finite. 

\begin{definition}[{see e.g., \cite[Section~3]{BG1}}]
\label{def_conic}
We say that a divisorial ideal $T(\bfa)$ is \emph{conic} if there exist $\bfx\in\sfM_\RR$ such that $\bfa=\ulcorner\sigma(\bfx)\urcorner$. 
\end{definition}

In the rest, we denote the set of isomorphism classes of conic divisorial ideals of a toric ring $R$ by $\calC(R)$, especially this is a finite set. 
An advantage to considering conic classes is that we can obtain a non-commutative ring having finite global dimension. 

\begin{theorem}
\label{motivation_thm}
Let $R$ be a toric ring as above. Concerning a conic divisorial ideal, we have the followings: 
\begin{enumerate}[\rm (1)]
\setlength{\parskip}{0pt} 
\setlength{\itemsep}{3pt}
\item  {\rm $($\cite[Proposition~3.6]{BG1}, \cite[Proposition~3.2.3]{SmVdB}$)$} 
Conic divisorial ideals are precisely modules appearing in $R^{1/m}$ 
as direct summands for $m\gg 0$, where $R^{1/m}=\kk[\tau^\vee\cap\frac{1}{m}\sfM]$ is the $R$-module consisting of $m$-th roots of elements in $R$. 
\item  {\rm $($\cite[Proposition~1.8]{SpVdB}, \cite[Corollary~6.2]{FMS}$)$} For $m\gg 0$, $\End_R(R^{1/m})$ has finite global dimension, 
that is, $\End_R(R^{1/m})$ is a non-commutative resolution $($= NCR$)$ of $R$ $($see Definition~\ref{def_NCCR}$)$. 
\end{enumerate}
\end{theorem}

A ring having finite global dimension appeared in M. Auslander's works e.g., \cite{Aus} 
and has been well studied in representation theory of algebras. Also, it is related with the dimension of derived categories \cite{Rou}. 
In this way, conic divisorial ideals have nice properties. 
Also, we remark that if $\mathrm{char}\,\kk=p$, the $R$-module $R^{1/p}$ can be obtained via the Frobenius morphism 
and the structure of $R^{1/p}$ is important in positive characteristic commutative algebra. 
Thus, our first interest lies on the classification of conic classes. 

In this paper, we will give the precise description of conic divisorial ideals of a Hibi ring, which is a toric ring constructed from a partially ordered set (= poset). 
We here review our main results (see Section~\ref{sec_conic} for further details on terminologies). 
Let $R=\kk[P]$ be a Hibi ring associated with a poset $P$ (see Subsection~\ref{subsec_Hibi}). 
As we will see in Subsection~\ref{subsec_conic_Hibi}, the class group of $R$ is $\Cl(R)\cong\ZZ^{n-d}$ 
where $d-1$ is the number of elements in $P$ and $n$ is the number of edges of the Hasse diagram of $\widehat{P}$, 
and hence each divisorial ideal is represented as $(a_1,\cdots, a_{n-d})\in\mathbb{Z}^{n-d}$. 
Then, conic divisorial ideals are characterized as follows, and hence these are precisely modules appearing in $R^{1/m}$ for sufficiently large $m\gg 0$. 

\begin{theorem}[{see Theorem~\ref{thm:conic}}] 
\label{conic_thm_intro}
Let $R$ be a Hibi ring associated with a poset $P$. 
Then, $(a_1,\cdots,a_{n-d})\in\mathbb{Z}^{n-d}$ corresponds to a conic divisorial ideal if and only if $(a_1,\cdots,a_{n-d})\in\calC(P)\cap\ZZ^{n-d}$
$($see $(\ref{ccp})$ for the precise definition of $\calC(P)$$)$. 
\end{theorem}

On the other hand, Van den Bergh introduced the notion of \emph{a non-commutative crepant resolution} (abbreviated as NCCR, see Definition~\ref{def_NCCR}), which is stronger than an NCR. 
For some nice singularities, an NCCR is derived equivalent to the usual crepant resolution, 
and hence this gives a new interaction between algebraic geometry, commutative algebra, and representation theory of algebras (see e.g., \cite{BKR,VdB2,Wem}). 
Subsequently, relationships with cluster tilting theory and their variants have been discovered (see e.g., \cite{Iya,IR,IW1}).  

One of the important problems is the existence of NCCRs for a given singularity. 
Here, we only mention results regarding the existence of NCCRs of toric rings. 
A toric ring defined by a simplicial cone is considered as a quotient singularity associated with a finite abelian group, in which case an NCCR is given by the skew group algebra (see e.g., \cite{VdB3,IW1}), and this is equivalent to the condition that a given ring admits a ``steady splitting" NC(C)R \cite{IN}. 
It is also known that a toric ring whose class group is $\ZZ$ has an NCCR \cite{VdB3}. 
For a $3$-dimensional Gorenstein toric ring, NCCRs can be obtained from consistent dimer models (see \cite{Bro,IU}, and also \cite{SpVdB3}). 
Concerning the existence of NCCRs for higher dimensional toric rings, there are several results e.g., \cite{SpVdB,SpVdB2}, but it is still open in general. 

As we mentioned in Theorem~\ref{motivation_thm} (2), the endomorphism ring of the direct sum of all conic divisorial ideals is an NCR. 
However, it is not an NCCR in general. In \cite{SpVdB}, some NCCRs of higher dimensional toric rings have been constructed under the assumption of ``quasi-symmetric", and the idea used in such a construction is to consider a part of conic classes. 
This strategy is also valid for some Hibi rings even if it is not quasi-symmetric as we will see later. 
Since we know the precise description of conic divisorial ideals of Hibi rings as in Theorem~\ref{conic_thm_intro}, we apply this description for constructing NCCRs. 
In this paper, we especially consider NCCRs for the Segre product of polynomial rings, which is the coordinate ring of the Segre embedding, and this is realized as a Hibi ring (see Example~\ref{segre}). 
Thus, this ring is quite important in both of commutative algebra and algebraic geometry. 
Using our observation on conic divisorial ideals and the method developed in \cite{SpVdB}, we show the following theorem. 
(We remark that this $R$ is not quasi-symmetric except the case of $t=2$.)

\begin{theorem}[{see Theorem~\ref{NCCR_Segre}}]
Let $R=S_1\# S_2\#\cdots\# S_t$ be the Segre product of $r$-dimensional polynomial rings $S_1, \cdots, S_t$ with $r\ge2$. 
Then, $R$ admits an NCCR $\End_R(M)$ where $M$ is the direct sum of some conic divisorial ideals. 
\end{theorem}

Furthermore, by applying the operation so-called ``mutation" to the NCCR shown in the above theorem, we can obtain other modules giving NCCRs of $R$, 
whose endomorphism rings are all derived equivalent (see Section~\ref{sec_mutation}).

\medskip

The content of this paper is the following. In Section~\ref{sec_conic}, we review some facts regarding conic divisorial ideals. 
We then give the precise description of conic divisorial ideals for Hibi rings in terms of their associated poset structure (see Theorem~\ref{thm:conic}). 
In Section~\ref{sec_NCCR}, we pay attention to the Segre product of polynomial rings and give an NCCR of it (see Theorem~\ref{NCCR_Segre}). 
In Section~\ref{sec_mutation}, we introduce the mutation, and apply it to the NCCR obtained in Section~\ref{sec_NCCR}. 

\subsection*{Notations and Conventions} 
Throughout this paper, we suppose that $\kk$ is an algebraically closed field. 
We denote 
by $\add_RM$ the category consisting of direct summands of finite direct sums of some copies of an $R$-module $M$, 
denote by $\refl(R)$ the category of reflexive $R$-modules. 
We denote by $\Cl(R)$ the class group of $R$, and let $\Cl(R)_\RR=\Cl(R)\otimes_\ZZ\RR$. 
We also denote by $\rmX(G)$ the character group of $G$, and let $\rmX(G)_\RR=\rmX(G)\otimes_\ZZ\RR$. 
Similarly, we denote by $\rmY(G)$ the group of one parameter subgroups of $G$, and let $\rmY(G)_\RR=\rmY(G)\otimes_\ZZ\RR$. 
In addition, we denote the $R$-dual functor by $(-)^*=\Hom_R(-,R)$.

\section{Conic divisorial ideals of Hibi rings}
\label{sec_conic}

\subsection{Preliminaries on conic divisorial ideals}
Let the notation be the same as in the previous section. 
In this subsection, we review some basic facts concerning divisorial ideals of a toric ring $R=\kk[\tau^\vee\cap\sfM]$ 
with $\tau=\mathrm{Cone}(v_1, \cdots, v_n)\subset\sfN_\RR$. 
As we mentioned before, any divisorial ideal takes the form $T(\bfa)$ with $\bfa\in\ZZ^n$. 
It is known that there is an exact sequence
\begin{equation}
\label{cl_seq}
0\rightarrow\sfM=\ZZ^d\xrightarrow{\sigma(-)}\ZZ^n\rightarrow\Cl(R)\rightarrow0, 
\end{equation}
and hence we see that for $\bfa, \bfa^\prime\in\ZZ^n$, $T(\bfa)\cong T(\bfa^\prime)$ 
if and only if there exists ${\bf y}\in\sfM$ such that $a_i=a_i^\prime+\sigma_i({\bf y})$ for all $i=1, \cdots, n$ 
(see e.g. \cite[Corollary~4.56]{BG2}). 
Let $\fkp_i\coloneqq T(\delta_{i1},\cdots,\delta_{in})$ where $\delta_{ij}$ is the Kronecker delta, 
and consider the prime divisor $\calD_i\coloneqq \rmV(\fkp_i)=\Spec R/\fkp_i$ on $\Spec R$. 
Then, the divisorial ideal $T(\bfa)=T(a_1,\cdots,a_n)$ corresponds to the Weil divisor $-(a_1\calD_1+\cdots +a_n\calD_n)$. 
In addition, the exact sequence (\ref{cl_seq}) gives the relations on these divisors. 
Precisely, a divisor $a_1\calD_1+\cdots +a_n\calD_n$ is zero in $\Cl(R)$ if $a_i=\langle\bfx,v_i\rangle$ for all $i=1,\cdots,n$ and some $\bfx\in\sfM$. 
In particular, if we take the $j$-th basic vector $\bfe_j\coloneqq(\delta_{j1},\cdots,\delta_{jd})\in\sfM$, 
we have that $$\langle\bfe_j,v_1\rangle D_1+\cdots+\langle\bfe_j,v_n\rangle D_n=0$$ in $\Cl(R)$ for all $j=1,\cdots,d$. 
Thus, we have that 
\begin{equation}
\label{relation_divisor}
v_{1,j}\calD_1+\cdots+v_{n,j}\calD_n=0 
\end{equation} 
for all $j=1,\cdots,d$ where $v_i\coloneqq(v_{i,1},\cdots,v_{i,d})\in\ZZ^d$. 
A divisorial ideal $T(\bfa)$ is called \emph{conic} if we can take ${\bf x}\in\sfM_\RR$ satisfying $\bfa=\ulcorner\sigma({\bf x})\urcorner$, 
equivalently there exists ${\bf x}\in\sfM_\RR$ such that $a_i-1<\sigma_i({\bf x})\le a_i$ for all $i=1,\cdots,n$. 
If ${\bf x}^\prime={\bf x}+{\bf y}$ with ${\bf y}\in\sfM$, then we see that $T(\sigma(\bfx^\prime))\cong T(\sigma(\bfx))$. 
Therefore, a conic divisorial ideal is determined by an element in $\sfM_\RR/\sfM$ up to isomorphism. 
We then discuss another description of conic divisorial ideals (cf. \cite[10.6]{SpVdB}). 
Using the exact sequence (\ref{cl_seq}), we see that 
\begin{equation*}
\sfM \cong \{(b_i)_i\in\ZZ^n \mid \sum_ib_i\calD_i=0 \ \text{in}\ \Cl(R) \}, 
\end{equation*}
\vspace{-0.3cm}
\begin{equation}
\label{relation_Rdiv}
\sfM_\RR \cong \{(b_i)_i\in\RR^n \mid \sum_ib_i\calD_i=0 \ \text{in}\ \Cl(R)_\RR \}, 
\end{equation}
where $\sigma_i(\bfx)=b_i$ for $\bfx\in\sfM$ and $i=1,\cdots, n$. 
Thus, we may consider $\bfu\in\sfM_\RR$ as $(u_i)_i\in\RR^n$ such that $\sum_iu_i\calD_i=0$ in $\Cl(R)_\RR$ where $\sigma_i(\bfu)=u_i$, 
and hence we may write 
\begin{equation*}
\TT(\sigma(\bfu))=\{(b_i)_i\in\ZZ^n \mid \sum_ib_i\calD_i=0 \ \text{in}\ \Cl(R),\ b_i\ge u_i \}. 
\end{equation*}
Here, we write $u_i=a_i+\delta_i$ with $a_i\in\ZZ$ and $\delta_i\in(-1,0]$, and set $c_i\coloneqq b_i-a_i\ge\delta_i$, $\alpha=-\sum_i a_i\calD_i\in\Cl(R)$. 
Then, we see that 
\begin{equation*}
\TT(\sigma(\bfu))=(a_i)_i+\{(c_i)_i\in\ZZ^n_{\ge0} \mid \sum_ic_i\calD_i=\alpha \}, 
\end{equation*}
and hence $T(\sigma(\bfu))$ is isomorphic to the divisorial ideal corresponding to $\alpha=-\sum_i a_i\calD_i$, which is $T(a_1,\cdots,a_n)$. 
Here, we note that $\ulcorner\sigma_i(\bfu)\urcorner=a_i$.  
Furthermore, we remark that $0=\sum_iu_i\calD_i=-\alpha+\sum_i\delta_i\calD_i$ in $\Cl(R)_\RR$, thus $\alpha=\sum_i\delta_i\calD_i$ in $\Cl(R)_\RR$. 
We easily follow the converse of this argument. 

As a conclusion of these observations, we have the following lemma. 

\begin{lemma}[{see also {\cite[Corollary 1.2]{Bru}}, \cite[Proposition~3.2.3]{SmVdB}}] 
\label{conic_characterization}
Let the notation be the same as above. 
There exists a one-to-one correspondence between the followings. 
\begin{enumerate}[\rm (1)]
\setlength{\parskip}{0pt} 
\setlength{\itemsep}{3pt}
\item A conic divisorial ideal $T(a_1,\cdots,a_n)$. 
\item An $\RR$-divisor $\sum_i\delta_i\calD_i$ with $(\delta_i)_i\in(-1,0]^n$ up to equivalence. 
Here, we say that two $\RR$-divisors are equivalent if their difference is in $\sfM_\RR$ $($see {\rm(\ref{relation_Rdiv})}$)$. 
\item A full-dimensional cell of the decomposition of the semi-open cube $(-1,0]^d$ by 
hyperplanes $H_{i,m}=\{{\bf x}\in\sfM_\RR \mid \sigma_i({\bf x})=m \}$ for some $m \in\ZZ$ and $i=1,\cdots,n$.  
\end{enumerate}

Here, we identify the cell $\bigcap_{i=1}^n L_{i,a_i}$ with $T(a_1,\cdots,a_n)$, 
where $L_{i,a_i}=\{{\bfx} \in \sfM_\RR \mid a_i-1<\sigma_i({\bfx}) \leq a_i\}$. 
\end{lemma}

As we mentioned in Theorem~\ref{motivation_thm}, conic divisorial ideals are precisely modules appearing in $R^{1/m}$ 
as direct summands for $m\gg 0$. 
Since $R^{1/m}$ is an MCM $R$-module, a conic divisorial ideal is also an MCM $R$-module. 
We notice that the number of non-isomorphic conic divisorial ideals is finite because that of rank one MCM $R$-modules is finite \cite[Corollary~5.2]{BG1}. 
Also, there exists a divisorial ideal that is a rank one MCM module but not conic (see e.g., \cite{Bae,Bru}).  
If $\tau$ is simplicial, every divisorial ideal is conic, because a torsion element in $\operatorname{Cl}(R)$ is conic \cite[Theorem~3.2]{BG1}. 

\subsection{Preliminaries on Hibi rings}
\label{subsec_Hibi}
We now turn our attention to a certain toric ring which is called a Hibi ring. 
Let $P=\{p_1,\cdots,p_{d-1}\}$ be a finite partially ordered set equipped with a partial order $\prec$. 
Throughout the paper, we will call it a \emph{poset} for short. Let 
\begin{equation*}
\calO(P)=\{(x_1,\cdots,x_{d-1}) \in \RR^{d-1} \mid  \, x_i \geq x_j \text{ if }p_i \preceq p_j \text{ in }P, \;  \, 0 \leq x_i \leq 1 \text{ for } i=1,\cdots,d-1\}.
\end{equation*}
It is known that $\calO(P)$ is a lattice polytope \cite[Corollary 1.3]{Sta2}, called the {\em order polytope} of a poset $P$. 

For a poset $P$, let $\kk[P]$ denote a toric ring defined by setting 
$$\kk[P]=\kk[{\bf X}^\alpha Y^n \mid \alpha \in n\calO(P) \cap \ZZ^{d-1}, \; n \in \ZZ_{\geq 0}],$$
where ${\bf X}^\alpha=X_1^{\alpha_1}\cdots X_{d-1}^{\alpha_{d-1}}$ for $\alpha=(\alpha_1,\cdots,\alpha_{d-1}) \in \ZZ^{d-1}$. 
This $\kk$-algebra is called the {\em Hibi ring} associated with $P$. The followings are some fundamental properties on Hibi rings, 
which were originally proved in \cite{Hibi}: 
\begin{itemize}
\setlength{\parskip}{0pt} 
\setlength{\itemsep}{3pt}
\item $\dim \kk[P]=|P|+1$; 
\item $\kk[P]$ is a standard graded CM normal domain, where the grading is defined by 
$\deg ({\bf X}^\alpha Y^n) = n$ for $\alpha \in n \calO(P) \cap \ZZ^{d-1}$; 
\item $\kk[P]$ is an algebra with straightening laws on $P$. 
\item It is known that $\kk[P]$ is Gorenstein if and only if $P$ is pure, 
where we say that $P$ is \emph{pure} if all of the maximal chains $p_{i_1} \prec \cdots \prec p_{i_\ell}$ have the same length. 
\end{itemize}

\begin{remark}
Originally, in the paper \cite{Hibi}, Hibi studied a quotient of a polynomial ring which looks like 
$\kk[X_\alpha \mid \alpha \in I(P)]\big/(X_\alpha X_\beta - X_{\alpha \cup \beta} X_{\alpha \cap \beta})$. 
Here, $I(P)$ denotes the set of all poset ideals of $P$, 
where a subset $I \subset P$ is called a poset ideal if $I$ satisfies that $p \in I$ and $p' \prec p$ imply $p' \in I$, 
and especially $I(P)$ is a distributive lattice with respect to the partial order defined by inclusion. 
In particular, he showed that this ring is a toric ring, and hence CM normal domain. 
On the other hand, our description $\kk[P]$ is based on the order polytope of $P$ studied in \cite{Sta2}, and of course $\kk[P]$ is
isomorphic to the original one. 
\end{remark}

Let $P=\{p_1,\cdots,p_{d-1}\}$. For $p_i, p_j \in P$, we say that $p_i$ {\em covers} $p_j$ 
if $p_j \prec p_i$ and there is no $p' \in P$ with $p_i \neq p'$ and $p_j \neq p'$ such that $p_j \prec p' \prec p_i$. 
Set $\widehat{P}=P \cup \{\hat{0}, \hat{1}\}$, 
where $\hat{0}$ (resp. $\hat{1}$) is the unique minimal (resp. maximal) element not belonging to $P$. 
Let us denote $p_0=\hat{0}$ and $p_d=\hat{1}$. We say that $e=\{p_i,p_j\}$, where $0 \leq i \not= j \leq d$, 
is an {\em edge} of $\widehat{P}$ if $e$ is an edge of the Hasse diagram of $\widehat{P}$ (i.e., $p_i$ covers $p_j$ or $p_j$ covers $p_i$). 
For each edge $e=\{p_i,p_j\}$ of $\widehat{P}$ with $p_i \prec p_j$, 
let $\sigma_e$ be a linear form in $\RR^d$ defined by 
\begin{align*}
\sigma_e({\bf x})\coloneqq
\begin{cases}
x_i-x_j, \;&\text{ if }j \not= d, \\
x_i, \; &\text{ if }j=d 
\end{cases}
\end{align*}
for ${\bf x}=(x_0,x_1,\cdots,x_{d-1})$. Let $\tau_P=\mathrm{Cone}(\sigma_e \mid e \text{ is an edge of }\widehat{P}) \subset \sfN_\RR=\RR^d$. 
Then, it is known that $\kk[P]=\kk[\tau_P^\vee \cap \ZZ^d]$. 
Let $e_1,\cdots,e_n$ be all the edges of $\widehat{P}$. We set a linear form $\sigma:\RR^d \rightarrow \RR^n$ by 
$$\sigma({\bf x})=(\sigma_{e_1}({\bf x}),\cdots,\sigma_{e_n}({\bf x})) \in \RR^n$$ 
for ${\bf x} \in \RR^d$. 

\subsection{Conic divisorial ideals of Hibi rings}
\label{subsec_conic_Hibi}

In this subsection, we consider conic divisorial ideals of Hibi rings. We retain the notations of the previous subsection. 

First, in order to discuss the class group of a Hibi ring, we prepare some terminologies. 
We say that a sequence $C=(p_{k_1},\cdots,p_{k_m})$ is a {\em cycle} in $\widehat{P}$ if $C$ forms a cycle 
in the Hasse diagram of $\widehat{P}$ (i.e., $p_{k_i} \not= p_{k_j}$ for $1 \leq i \not= j \leq m$ 
and each $\{p_{k_i},p_{k_{i+1}}\}$ is an edge of $\widehat{P}$ for $1 \leq i \leq m$, where $p_{k_{m+1}}=p_{k_1}$). 
Moreover, we say that a cycle $C$ is a {\em circuit} if $\{p_{k_i},p_{k_j}\}$ is not an edge of $\widehat{P}$ 
for any $1 \leq i,j \leq m$ with $|i-j| \geq 2$. 
We say that a set of $d$ edges $e_{i_1},\cdots,e_{i_d}$ of $\widehat{P}$ is a {\em spanning tree} if 
they form a spanning tree of the Hasse diagram of $\widehat{P}$, 
that is, any element in $\widehat{P}$ is an endpoint of some edge $e_{i_j}$ and the edges $e_{i_1},\cdots,e_{i_d}$ do not form cycles. 
We remark that a spanning tree is not unique. 
For $p \in \widehat{P} \setminus \{\hat{1}\}$, let $U(p)$ denote the set of all elements in $\widehat{P}$ which cover $p$. 
Similarly, for $p \in \widehat{P} \setminus \{\hat{0}\}$, let $D(p)$ denote the set of all elements in $\widehat{P}$ which are covered by $p$. 

Then, we discuss the class group $\Cl(\kk[P])$ for describing conic divisorial ideals of Hibi rings. 
By definition of $\sigma(-)$ and (\ref{relation_divisor}), we see that 
the prime divisor $\calD_e$ indexed by the edges $e$ of $\widehat{P}$ satisfies the relations: 
\begin{align}
\label{relation_div_hibi}
\sum_{q \in U(p)} \calD_{\{q,p\}}=\sum_{q' \in D(p)} \calD_{\{p,q'\}} \text{ for }p \in \widehat{P} \setminus \{\hat{0}, \hat{1}\}, 
\; \text{ and }\; \sum_{q \in U(\hat{0})}\calD_{\{q,p_0\}}=0. 
\end{align}
In particular, we can take prime divisors corresponding to edges not contained in a spanning tree as generators of $\Cl(\kk[P])$, 
thus we have that $\Cl(\kk[P]) \cong \ZZ^n/\sigma(\ZZ^d)\cong  \ZZ^{n-d}$ (see \cite{HHN}). 
More precisely, we fix a spanning tree of $\widehat{P}$ and let $e_1,\cdots,e_d$ be its edges for simplicity of notation. 
Thus, let $e_{d+1},\cdots,e_n$ be the remaining edges of $\widehat{P}$. 
Then, using (\ref{relation_div_hibi}) any Weil divisor can be described as $\sum_{i=1}^{n-d}a_i\calD_{e_{d+i}}$, and we identify this with $(a_1,\cdots,a_{n-d})\in\ZZ^{n-d}$. 

In what follows, we will give a correspondence of which elements in $\Cl(\kk[P])\cong \ZZ^{n-d}$ describe the conic divisorial ideals. 
For a cycle $C=(p_{k_1},\cdots,p_{k_m})$ in $\widehat{P}$, let 
\begin{align*}
X_C^+&=\{\{p_{k_i},p_{k_{i+1}}\} \mid 1 \leq i \leq m, \;  p_{k_i} \prec p_{k_{i+1}}\}, \\
X_C^-&=\{\{p_{k_i},p_{k_{i+1}}\} \mid 1 \leq i \leq m, \; p_{k_{i+1}} \prec p_{k_i}\}, \\
Y_C^{\pm}&=X_C^{\pm} \cap \{e_1,\cdots,e_d\}, \text{ and }\; 
Z_C^{\pm} = X_C^\pm \cap \{e_{d+1},\cdots,e_n\}, 
\end{align*}
where $p_{k_{m+1}}=p_{k_1}$. 

\begin{example}
For the Hasse diagram of $\widehat{P}$ shown in the left hand side of Figure~\ref{ex_Xc}, we fix the spanning tree $\{e_2, e_3,\cdots,e_8 \}$. 
For the cycle $C=(p_1,p_2,p_3,p_5,p_4,p_0)$, we have that 
\begin{center}
\begin{tabular}{lll}
$X^+_C=\{e_1, e_2, e_3 \}$, & $Y^+_C=\{e_2, e_3 \}$, & $Z^+_C=\{e_1 \}$, \\
$X^-_C=\{e_5, e_6, e_9 \}$, & $Y^-_C=\{e_5, e_6 \}$,& $Z^-_C=\{e_9 \}$. 
\end{tabular}
\end{center}

\begin{figure}[H]
\begin{center}
{\scalebox{0.4}{
\begin{tikzpicture}
\node (N1) at (0,0)
{\scalebox{1}{
\begin{tikzpicture}
\coordinate (Min) at (0,0); \coordinate (Max) at (0,8);
\coordinate (N11) at (-2,2); \coordinate (N12) at (-2,4); \coordinate (N13) at (-2,6); 
\coordinate (N21) at (2,2); \coordinate (N22) at (2,4); \coordinate (N23) at (2,6);
\node[blue] at (0,-0.8) {\huge$\widehat{0}=p_0$} ; \node at (0,8.8) {\huge$\widehat{1}=p_7$} ; 

\draw[line width=0.07cm]  (Min)--(N11); \node at (-1.5,0.8) {\huge$e_1$} ; 
\draw[line width=0.07cm,red]  (N11)--(N12) ; \node[red] at (-2.45,3) {\huge$e_2$} ; 
\draw[line width=0.07cm,red]  (N12)--(N13) ; \node[red] at (-2.45,5) {\huge$e_3$} ; 
\draw[line width=0.07cm,red]  (N13)--(Max) ; \node[red] at (-1.5,7.2) {\huge$e_4$} ; 
\draw[line width=0.07cm,red]  (Min)--(N21) ; \node[red] at (1.5,0.8) {\huge$e_5$} ; 
\draw[line width=0.07cm,red]  (N21)--(N22) ; \node[red] at (2.45,3) {\huge$e_6$} ; 
\draw[line width=0.07cm,red]  (N22)--(N23) ; \node[red] at (2.45,5) {\huge$e_7$} ; 
\draw[line width=0.07cm,red]  (N23)--(Max) ; \node[red] at (1.5,7.2) {\huge$e_8$} ; 
\draw[line width=0.07cm]  (N13)--(N22) ; \node at (0,5.5) {\huge$e_9$} ; 
\draw[line width=0.05cm, dashed, blue]  (-1.6,2)--(-1.6,4)--(-1.6,5.5)--(1.6,3.9)--(1.6,2)--(0,0.4)--(-1.6,2) ;

\node[blue] at (0,3) {\huge$C$} ;
\node[blue] at (-2.65,2) {\huge$p_1$} ; \node[blue] at (-2.65,4) {\huge$p_2$} ; \node[blue] at (-2.65,6) {\huge$p_3$} ; 
\node[blue] at (2.65,2) {\huge$p_4$} ; \node[blue] at (2.65,4) {\huge$p_5$} ; \node at (2.65,6) {\huge$p_6$} ;

\draw [line width=0.07cm, fill=lightgray] (Min) circle [radius=0.2] ; 
\draw [line width=0.07cm, fill=lightgray] (Max) circle [radius=0.2] ; 
\draw [line width=0.07cm, fill=white] (N11) circle [radius=0.2] ; \draw [line width=0.07cm, fill=white] (N12) circle [radius=0.2] ;
\draw [line width=0.07cm, fill=white] (N13) circle [radius=0.2] ; 
\draw [line width=0.07cm, fill=white] (N21) circle [radius=0.2] ; \draw [line width=0.07cm, fill=white] (N22) circle [radius=0.2] ;
\draw [line width=0.07cm, fill=white] (N23) circle [radius=0.2] ; 
\end{tikzpicture} }} ; 

\node (N2) at (8.5,0)
{\scalebox{1}{
\begin{tikzpicture}
\node at (0,-0.8) {\huge$\widehat{0}$} ; \node at (0,8.8) {\huge$\widehat{1}$} ; 

\draw[line width=0.6cm,gray]  (-1.7,5.85)--(N22)--(N21)--(0.2,0.2) ; 
\draw[line width=0.6cm,lightgray]  (-0.2,0.2)--(N11)--(N12)--(-2,5.7) ; 

\draw[line width=0.07cm]  (Min)--(N11) ; \draw[line width=0.07cm]  (N11)--(N12) ; 
\draw[line width=0.07cm]  (N12)--(N13) ; \draw[line width=0.07cm]  (N13)--(Max) ; 
\draw[line width=0.07cm]  (Min)--(N21) ; \draw[line width=0.07cm]  (N21)--(N22) ; 
\draw[line width=0.07cm]  (N22)--(N23) ; \draw[line width=0.07cm]  (N23)--(Max) ; 
\draw[line width=0.07cm]  (N13)--(N22) ; 

\draw [line width=0.07cm, fill=lightgray] (Min) circle [radius=0.2] ; 
\draw [line width=0.07cm, fill=lightgray] (Max) circle [radius=0.2] ; 
\draw [line width=0.07cm, fill=white] (N11) circle [radius=0.2] ; \draw [line width=0.07cm, fill=white] (N12) circle [radius=0.2] ;
\draw [line width=0.07cm, fill=white] (N13) circle [radius=0.2] ; 
\draw [line width=0.07cm, fill=white] (N21) circle [radius=0.2] ; \draw [line width=0.07cm, fill=white] (N22) circle [radius=0.2] ;
\draw [line width=0.07cm, fill=white] (N23) circle [radius=0.2] ; 
\node at (-3,3) {\huge$X_C^+$} ; \node at (3,3) {\huge$X_C^-$} ;
\end{tikzpicture} }} ; 

\node (N3) at (17,0)
{\scalebox{1}{
\begin{tikzpicture}
\node at (0,-0.8) {\huge$\widehat{0}$} ; \node at (0,8.8) {\huge$\widehat{1}$} ; 
\node at (-3,4) {\huge$Y_C^+$} ; \node at (3,2) {\huge$Y_C^-$} ;
\node at (-1.5,0.3) {\huge$Z_C^+$} ; \node at (0,6) {\huge$Z_C^-$} ;

\draw[line width=0.6cm,pink]  (-0.2,0.2)--(-1.8,1.8) ; 
\draw[line width=0.6cm,pink]  (-1.7,5.85)--(1.7,4.15) ; 
\draw[line width=0.6cm,red]  (-2,2.3)--(-2,5.7) ; 
\draw[line width=0.6cm,red]  (2,3.7)--(N21)--(0.2,0.2) ; 

\draw[line width=0.07cm]  (Min)--(N11) ; \draw[line width=0.07cm]  (N11)--(N12) ; 
\draw[line width=0.07cm]  (N12)--(N13) ; \draw[line width=0.07cm]  (N13)--(Max) ; 
\draw[line width=0.07cm]  (Min)--(N21) ; \draw[line width=0.07cm]  (N21)--(N22) ; 
\draw[line width=0.07cm]  (N22)--(N23) ; \draw[line width=0.07cm]  (N23)--(Max) ; 
\draw[line width=0.07cm]  (N13)--(N22) ; 

\draw [line width=0.07cm, fill=lightgray] (Min) circle [radius=0.2] ; 
\draw [line width=0.07cm, fill=lightgray] (Max) circle [radius=0.2] ; 
\draw [line width=0.07cm, fill=white] (N11) circle [radius=0.2] ; \draw [line width=0.07cm, fill=white] (N12) circle [radius=0.2] ;
\draw [line width=0.07cm, fill=white] (N13) circle [radius=0.2] ; 
\draw [line width=0.07cm, fill=white] (N21) circle [radius=0.2] ; \draw [line width=0.07cm, fill=white] (N22) circle [radius=0.2] ;
\draw [line width=0.07cm, fill=white] (N23) circle [radius=0.2] ; 
\end{tikzpicture} }} ; 
\end{tikzpicture}
} }
\end{center}
\caption{Example of $X_C^{\pm}, Y_C^{\pm}$, and $Z_C^{\pm}$}
\label{ex_Xc}
\end{figure}
\end{example}

Let $\calC(P)$ be a convex polytope defined by 
\begin{equation}
\label{ccp}
\calC(P)=\Bigg\{(z_1,\cdots,z_{n-d}) \in \RR^{n-d} \,\Big| 
-|X_C^-|+1 \leq \sum_{e_{d+\ell} \in Z_C^+}  z_\ell -\sum_{e_{d+\ell'} \in Z_C^-}  z_{\ell'}\leq |X_C^+|-1 \Bigg\},
\end{equation}
where $C=(p_{k_1},\cdots,p_{k_m})$ runs over all circuits in $\widehat{P}$. 

Now we are ready to describe the conic divisorial ideals of a Hibi ring. 
\begin{theorem}\label{thm:conic}
Let the notation be the same as above. 
Then, each point $(z_1,\cdots,z_{n-d}) \in \calC(P) \cap \ZZ^{n-d}$ one-to-one corresponds to the conic divisorial ideal of $\kk[P]$. 
\end{theorem}
\begin{remark}
The expression of a divisorial ideal as the element of $\Cl(\kk[P])$ depends on a spanning tree. Thus, if we choose another spanning tree, the shape of $\calC(P)$ might be different. 
However, we have a bijection between such different expression via the relations (\ref{relation_divisor}). Thus, it does not matter how we choose a spanning tree. 
\end{remark}
\begin{proof}[Proof of Theorem \ref{thm:conic}]

Let us fix a spanning tree $e_1,\cdots,e_d$ of $\widehat{P}$. 
We sometimes denote the linear form $\sigma_{e_i}$ by $\sigma_i$ for simplicity. 
We see that there exists a unimodular transformation $\phi:\RR^d \rightarrow \RR^d$ 
such that $$\phi((-1,0]^d) = \{{\bf x} \in \RR^d \mid -1 < \sigma_i({\bf x}) \leq 0, i=1,\cdots,d\}.$$ 
In fact, when we regard each $\sigma_i$ as a vector of $\ZZ^d$, 
there is at most one $1$ and at most one $-1$ in the entry of each $\sigma_i$ and other entries are all $0$. 
Then the matrix whose row vectors are $\sigma_1,\cdots,\sigma_d$ becomes a totally unimodular matrix. 
Moreover, since $e_1,\cdots,e_d$ form a spanning tree of $\widehat{P}$, we conclude that 
the matrix has the determinant $\pm 1$, i.e., is a unimodular matrix. We may define $\phi$ by this matrix. 
Since this $\phi$ does not change a given Hibi ring up to isomorphism, we will work with the linear form $\sigma$ 
satisfying $-1 < \sigma_i({\bf x}) \leq 0$ for all $i=1,\cdots,d$. 

\medskip

Using the one-to-one correspondence mentioned in Lemma~\ref{conic_characterization}, 
we first show that $(m_1,\cdots,m_n) \in \ZZ^n$ that represents a conic divisorial ideal $T(m_1,\cdots, m_n)$ is contained in 
$\calC(P)\cap\ZZ^{n-d}$ as the element of $\Cl(R)$. 
To show this, we consider the decomposition of the semi-open cube $(-1,0]^d$ by the hyperplanes 
$H_{i,m} = \{{\bf x} \in \RR^d \mid \sigma_i({\bf x}) = m\}$ for $1 \leq i \leq n$ with $m \in \ZZ$. 
Let $L_{i,m}=\{{\bf x} \in \RR^d \mid m-1<\sigma_i({\bf x}) \leq m \}$. 
Since each full-dimensional cell of the decomposition is of the form $\bigcap_{i=1}^n L_{i,m_i}$, 
where $m_i \in \ZZ$ for each $1 \leq i \leq n$, 
we identify each full-dimensional cell $\bigcap_{i=1}^n L_{i,m_i}$ of the decomposition with $(m_1,\cdots,m_n) \in \ZZ^n$.

\medskip

\noindent({\bf The first step}): 
Let $\Pi=\{{\bf x} \in \RR^d \mid -1 < \sigma_i({\bf x}) \leq 0, i=1,\cdots,d\}$. 
Notice that the intersection of $\Pi$ and $H_{i,m}$ for $1 \leq i \leq d$ is just the boundary of $\Pi$ if they intersect. 
Hence, in what follows, we may discuss the decomposition of $\Pi$ by the hyperplanes $H_{i,m}$ for $d+1 \leq i \leq n$ with $m \in \ZZ$. 
In other words, we only consider the cells $(m_1,\cdots,m_n) \in \ZZ^n$ 
which are of the form $(\underbrace{0,\cdots,0}_d,m_{d+1},\cdots,m_n) \in \ZZ^n$. 

\noindent({\bf The second step}): 
We determine the possible $(m_{d+1},\cdots,m_n) \in \ZZ^{n-d}$ such that $\Pi \cap \bigcap_{i=d+1}^n L_{i,m_i}$ defines a full-dimensional cell. 

Fix $e=e_i \in \{e_{d+1},\cdots,e_n\}$. Since $e_1,\cdots,e_d$ form a spanning tree of $\widehat{P}$, 
there should exist a unique cycle $C=(p_{k_1},\cdots,p_{k_m})$ in $\widehat{P}$ 
with $e=\{p_{k_j},p_{k_{j+1}}\}$ for some $1 \leq j \leq m$ and $\{p_{k_\ell},p_{k_{\ell+1}}\} \in \{e_1,\cdots,e_d\}$ 
for every $1 \leq \ell \leq m$ with $\ell \neq j$, where $p_{k_{m+1}}=p_{k_1}$. Namely, $Z_C^+ \cup Z_C^- = \{\{p_{k_j},p_{k_{j+1}}\}\}$. 
Without loss of generality, we assume $e \in Z_C^+$. 
Now, we can easily see that $\sum_{f \in X_C^+}\sigma_f = \sum_{f' \in X_C^-} \sigma_{f'}$ by the definition of $\sigma_f,\sigma_{f'}$ and $X_C^\pm$. 
Thus, we obtain $\sigma_e=\sum_{f' \in Y_C^-}\sigma_{f'} - \sum_{f \in Y_C^+} \sigma_f$. 
Recall that for each ${\bf x} \in \Pi$, one has $-1<\sigma_i({\bf x}) \leq 0$ for $1 \leq i \leq d$. 
Thus, we obtain that $-|Y_C^-| < \sigma_e({\bf x}) < |Y_C^+|$. Hence, there exists ${\bf x} \in \Pi$ with $m-1<\sigma_e({\bf x}) \leq m$ 
(i.e., ${\bf x} \in \Pi \cap L_{i,m}$) if and only if $-|Y_C^-| +1 \leq m \leq |Y_C^+|$. 
This means that each cell corresponding to $(0,\cdots,0,m_{d+1},\cdots,m_n) \in \ZZ^n$ satisfies 
$-|Y_C^-|+1=-|X_C^-|+1 \leq m_i \leq |X_C^+|-1=|Y_C^+|$. 

\medskip

\noindent({\bf The third step}): 
We then consider a cycle $C=(p_{k_1},\cdots,p_{k_m})$ in $\widehat{P}$ with $|Z_C^+ \cup Z_C^-| \geq 2$, 
and hence a cycle $C$ is different from the one in the second step. 
Since $\sum_{f \in Y_C^+ \cup Z_C^+}\sigma_f({\bf x}) = \sum_{f' \in Y_C^- \cup Z_C^-}\sigma_{f'}({\bf x})$, 
one has $\sum_{f \in Z_C^+}\sigma_f({\bf x}) - \sum_{f' \in Z_C^-}\sigma_{f'}({\bf x})
=\sum_{f' \in Y_C^-}\sigma_{f'}({\bf x})-\sum_{f \in Y_C^+}\sigma_f({\bf x})$. 
Similar to the discussion in the second step, we have 
$$-|Y_C^-| < \sum_{f \in Z_C^+}\sigma_f({\bf x}) - \sum_{f' \in Z_C^-}\sigma_{f'}({\bf x}) < |Y_C^+|.$$
On the other hand, it follows from the second step that we have 
\begin{equation*}
\sum_{e_\ell \in Z_C^+}m_\ell - \sum_{e_{\ell'} \in Z_C^-}m_{\ell'}-|Z_C^+|< 
\sum_{f \in Z_C^+}\sigma_f({\bf x}) - \sum_{f' \in Z_C^-}\sigma_{f'}({\bf x}) 
 <\sum_{e_\ell \in Z_C^+}m_\ell - \sum_{e_{\ell'} \in Z_C^-}m_{\ell'}+|Z_C^-|.
\end{equation*}
Hence, by these two inequalites, we obtain that 
$$-|Y_C^-|-|Z_C^-|<\sum_{e_\ell \in Z_C^+}m_\ell - \sum_{e_{\ell'} \in Z_C^-}m_{\ell'}<|Y_C^+|+|Z_C^+|,$$
which is equivalent to 
$$-|X_C^-|+1 \leq \sum_{e_\ell \in Z_C^+}m_\ell - \sum_{e_{\ell'} \in Z_C^-}m_{\ell'}\leq |X_C^+|-1.$$ 
This means that each cell $(0,\cdots,0,m_{d+1},\cdots,m_n) \in \ZZ^n$ satisfies this inequality. 

\medskip

\noindent({\bf The fourth step}): 
We will prove that each inequality for a cycle that is not a circuit always comes from the ones for circuits.
This means that we may only consider the inequalites obtained in the second and third steps for circuits. 
Let $C=(p_{k_1},\ldots,p_{k_m})$ be a cycle in $\widehat{P}$ that is not a circuit. 
Thus, there exists an edge $\{p_{k_a},p_{k_b}\}$ of $\widehat{P}$ such that $1 \leq a < b \leq m$ with $b-a > 1$ and $(a,b) \neq (1,m)$. 
Assume $p_{k_a} \prec p_{k_b}$. Let $C_1=(p_{k_1},\ldots,p_{k_a},p_{k_b},\ldots,p_{k_m})$ and $C_2=(p_{k_a},p_{k_{a+1}},\ldots,p_{k_b})$. 
Then $C_1$ and $C_2$ are cycles and we see that $\{p_{k_a},p_{k_b}\} \in X_{C_1}^+ \cap X_{C_2}^-$, 
$X_C^+= X_{C_1}^+ \cup X_{C_2}^+ \setminus \{\{p_{k_a},p_{k_b}\}\}$ and $X_C^-= X_{C_1}^- \cup X_{C_2}^- \setminus \{\{p_{k_a},p_{k_b}\}\}$. 
Thus, $|X_C^+|=|X_{C_1}^+|+|X_{C_2}^+|-1$ and $|X_C^-|=|X_{C_1}^-|+|X_{C_2}^-|-1$. Hence, 
by summing up two inequalities 
\begin{align*}
&-|X_{C_1}^-|+1 \leq \sum_{e_{d+\ell} \in Z_{C_1}^+}  z_\ell - \sum_{e_{d+\ell'} \in Z_{C_1}^-}  z_{\ell'} \leq |X_{C_1}^+|-1 \text{ and } \\
&-|X_{C_2}^-|+1 \leq \sum_{e_{d+\ell} \in Z_{C_2}^+}  z_\ell - \sum_{e_{d+\ell'} \in Z_{C_2}^-}  z_{\ell'} \leq |X_{C_2}^+|-1, 
\end{align*}
we obtain the inequality for $C$. 

\medskip

\noindent({\bf The fifth step}): 
Therefore, from the second and the third steps together with the fourth step, 
we see that if each $(0,\cdots,0,m_{d+1},\cdots,m_n) \in \ZZ^n$ defines a full-dimensional cell, 
then $(z_1,\cdots,z_{n-d})$ satisfies the conditions described in \eqref{ccp} as required, 
by identifying each $(0,\cdots,0,m_{d+1},\cdots,m_n) \in \ZZ^n$ with $(z_1,\cdots,z_{n-d}) \in \ZZ^{n-d}$ by $m_{d+j} = z_j$. 

\medskip

Next, we show the converse, that is, we show that $T(m_1,\cdots, m_n)$ contained in 
$\calC(P)\cap\ZZ^{n-d}$ as the element of $\Cl(R)$ actually determines a conic one. 
In order to achieve this, we identify $T(m_1,\cdots, m_n)$ with the divisor $-\sum^n_{i=1}m_i\calD_i$, 
and show that $-m_i\in(-1,0]$ for any $i=1,\cdots, n$ when we consider $-\sum m_i\calD_i$ as the element of $\Cl(R)_\RR$ (see Lemma~\ref{conic_characterization}). 
Let $\sum^n_{i=1}\overline{m}_i\calD_i$ be the divisor in $\Cl(R)_\RR$ corresponding to $-\sum^n_{i=1}m_i\calD_i$. 
Thus, $m_i+\overline{m}_i=\sigma_i(\bfx)$ holds for any $i=1,\cdots, n$ and some $\bfx\in\RR^d$. 
Using the unimodular transformation $\phi$ discussed in the beginning of this proof, we may assume that 
$\sigma_i(\bfx)=x_i$ for $i=1,\cdots,d$ where $\bfx=(x_1,\cdots,x_d)$. 
Thus, if we take $\bfx\in\RR^d$ such that $m_i=\ulcorner x_i \urcorner$, then $x_i-m_i=\overline{m}_i\in(-1,0]$ holds for all $i=1,\cdots,d$. 
Since conic divisorial ideals are determined by elements in $\sfM_\RR/\sfM$, we may assume $x_i=\overline{m}_i$. 
In the following, we use this $\bfx\in(-1,0]^d$, especially $-1<\sigma_i(\bfx)\le 0$ holds for all $i=1,\cdots,d$.

For any $e_j\in\{ e_{d+1},\cdots, e_n\}$, by the same argument as in the second step, we see that 
$m_j-1<\sigma_{e_j}({\bf x}) \leq m_j$ if and only if 
\begin{equation}
\label{ineq_C}
-|X_{C_j}^-| +1 \leq m_j \leq |X_{C_j}^+|-1
\end{equation}
 for a certain cycle $C_j$. 
Since we are considering $m_j$'s satisfying the condition (\ref{ccp}) for all circuits, the condition (\ref{ineq_C}) holds. 
(Recall that the inequality arising from a cycle $C$ is obtained by combining the ones arising from circuits as we noted in the fourth step.) 
Thus, we have that $m_j-1<\sigma_{e_j}({\bf x}) \leq m_j$, and hence $\overline{m}_j\in(-1,0]$ holds. 
\end{proof}

For the convenience, we describe the full-dimensional cell corresponding to $(m_1,\cdots,m_{n-d}) \in \calC(P) \cap \ZZ^{n-d}$. 
We set the region consisting of those $(y_1,\cdots,y_d) \in \RR^d$ satisfying the following inequalities: 
\begin{itemize}
\setlength{\parskip}{0pt} 
\setlength{\itemsep}{3pt}
\item $-1< y_i \leq 0$ for $1 \leq i \leq d$; 
\item $\displaystyle m_j-1 < \sum_{e_{\ell'} \in Y_C^-} y_{\ell'}-\sum_{e_\ell \in Y_C^+} y_\ell \leq m_j$ 
for $1 \leq j \leq n-d$, where $C=(p_{k_1},\cdots,p_{k_m})$ is a unique cycle in $\widehat{P}$ 
with $\{p_{k_m},p_{k_1}\}=e_{d+j}$, $p_{k_m} \prec p_{k_1}$ and $Z_C^+ \cup Z_C^- = \{e_{d+j}\}$. 
\end{itemize}
Then by the final part of the proof of Theorem \ref{thm:conic}, this is the full-dimensional cell corresponding to $(m_1,\cdots,m_{n-d})$. 

\begin{example}\label{segre}
Let $t \geq 2$. Given positive integers $r_1,\cdots,r_t$, 
we consider a disjoint union of chains, i.e., a poset $P=\{p_{i,j} \mid 1 \leq i \leq t, 1 \leq j \leq r_i\}$ 
equipped with the partial order $p_{i,1} \prec \cdots \prec p_{i,r_i}$ for each $1 \leq i \leq t$  (see Figure~\ref{segre_poset}). 
Let $e_{i,j}=\{p_{i,j-1},p_{i,j}\}$ be the edge of $\widehat{P}$ for $1 \leq i \leq t$ and $1 \leq j \leq r_i+1$, 
where $p_{i,0}$ (resp. $p_{i,r_i+1}$) denotes $\hat{0}$ (resp. $\hat{1}$) for every $i$. 
Then, we see that the Hibi ring associated with $P$ is isomorphic to the Segre product of polynomial rings $\#_{i=1}^t \kk[x_{i,1},\cdots,x_{i,r_i+1}]$, 
which is the ring generated by monomials of the form $x_{1,j_1}x_{2,j_2}\cdots x_{t,j_t}$ where $j_i=1,\cdots,r_i+1$ and $i=1,\cdots,t$. 
The set of the edges 
$$\{e_{i,j} \mid 1 \leq i \leq t, 2 \leq j \leq r_i+1\} \cup \{e_{t,1}\}$$ 
described by red line in Figure~\ref{segre_poset}
is a spanning tree of $\widehat{P}$. Fix this spanning tree. 
In this case, the class group is isomorphic to $\ZZ^{t-1}$ which is generated by divisors $\calD_{e_{1,1}},\cdots,\calD_{e_{t-1,1}}$. 

\begin{figure}[H]
\begin{center}
{\scalebox{0.55}{
\begin{tikzpicture}

\coordinate (Min) at (4.5,-2.5); \coordinate (Max) at (4.5,9.5);
\coordinate (N11) at (0,0); \coordinate (N12) at (0,1.5); \coordinate (N13) at (0,3); \coordinate (N14) at (0,5.5); \coordinate (N15) at (0,7);
\coordinate (N21) at (3,0); \coordinate (N22) at (3,1.5); \coordinate (N23) at (3,3); \coordinate (N24) at (3,5.5); \coordinate (N25) at (3,7);
\coordinate (Ntt1) at (6,0); \coordinate (Ntt2) at (6,1.5); \coordinate (Ntt3) at (6,3); \coordinate (Ntt4) at (6,5.5); \coordinate (Ntt5) at (6,7);
\coordinate (Nt1) at (9,0); \coordinate (Nt2) at (9,1.5); \coordinate (Nt3) at (9,3); \coordinate (Nt4) at (9,5.5); \coordinate (Nt5) at (9,7);

\draw[line width=0.07cm]  (N11)--(Min) node[midway,xshift=0cm,yshift=-0.4cm] {\Large $e_{1,1}$};
\draw[line width=0.07cm,red]  (N11)--(N12) node[midway,xshift=-0.5cm] {\Large $e_{1,2}$}; 
\draw[line width=0.07cm,red]  (N12)--(N13) node[midway,xshift=-0.5cm] {\Large $e_{1,3}$}; 
\draw[line width=0.07cm,red]  (N13)--(0,3.5) ; \draw[line width=0.07cm,red]  (0,5)--(N14); \draw[line width=0.07cm, loosely dotted,red]  (0,3.8)--(0,4.7);
\draw[line width=0.07cm,red]  (N14)--(N15) node[midway,xshift=-0.5cm] {\Large $e_{1,r_1}$}; 
\draw[line width=0.07cm,red]  (N15)--(Max) node[midway,xshift=-0.3cm,yshift=0.4cm] {\Large $e_{1,r_1+1}$};

\draw[line width=0.07cm]  (N21)--(Min) node[midway,xshift=-0.9cm,yshift=0.6cm] {\Large $e_{2,1}$};
\draw[line width=0.07cm,red]  (N21)--(N22) node[midway,xshift=-0.5cm] {\Large $e_{2,2}$}; 
\draw[line width=0.07cm,red]  (N22)--(N23) node[midway,xshift=-0.5cm] {\Large $e_{2,3}$}; 
\draw[line width=0.07cm,red]  (N23)--(3,3.5); \draw[line width=0.07cm,red]  (3,5)--(N24); \draw[line width=0.07cm, loosely dotted,red]  (3,3.8)--(3,4.7);
\draw[line width=0.07cm,red]  (N24)--(N25) node[midway,xshift=-0.5cm] {\Large $e_{2,r_2}$}; 
\draw[line width=0.07cm,red]  (N25)--(Max) node[midway,xshift=-1.1cm,yshift=-0.6cm] {\Large $e_{2,r_2+1}$}; 

\draw[line width=0.07cm]  (Ntt1)--(Min) node[midway,xshift=1.1cm,yshift=0.6cm] {\Large $e_{t-1,1}$};
\draw[line width=0.07cm,red]  (Ntt1)--(Ntt2) node[midway,xshift=0.65cm] {\Large $e_{t-1,2}$}; 
\draw[line width=0.07cm,red]  (Ntt2)--(Ntt3) node[midway,xshift=0.65cm] {\Large $e_{t-1,3}$}; 
\draw[line width=0.07cm,red]  (Ntt3)--(6,3.5); \draw[line width=0.07cm,red]  (6,5)--(Ntt4); \draw[line width=0.07cm, loosely dotted,red]  (6,3.8)--(6,4.7);
\draw[line width=0.07cm,red]  (Ntt4)--(Ntt5) node[midway,xshift=0.88cm] {\Large $e_{t-1,r_{t-1}}$}; 
\draw[line width=0.07cm,red]  (Ntt5)--(Max) node[midway,xshift=1.3cm,yshift=-0.6cm] {\Large $e_{t-1,r_{t-1}+1}$};

\draw[line width=0.07cm,red]  (Nt1)--(Min) node[midway,xshift=0cm,yshift=-0.4cm] {\Large $e_{t,1}$}; 
\draw[line width=0.07cm,red]  (Nt1)--(Nt2) node[midway,xshift=0.45cm] {\Large $e_{t,2}$}; 
\draw[line width=0.07cm,red]  (Nt2)--(Nt3) node[midway,xshift=0.45cm] {\Large $e_{t,3}$}; 
\draw[line width=0.07cm,red]  (Nt3)--(9,3.5); \draw[line width=0.07cm,red]  (9,5)--(Nt4); \draw[line width=0.07cm, loosely dotted,red]  (9,3.8)--(9,4.7);
\draw[line width=0.07cm,red]  (Nt4)--(Nt5) node[midway,xshift=0.5cm] {\Large $e_{t,r_t}$}; 
\draw[line width=0.07cm,red]  (Nt5)--(Max) node[midway,xshift=0.3cm,yshift=0.4cm] {\Large $e_{t,r_t+1}$}; 

\draw[line width=0.07cm, loosely dotted,red]  (4,4.25)--(5,4.25); 

\draw [line width=0.07cm, fill=lightgray] (Min) circle [radius=0.18] ; 
\draw [line width=0.07cm, fill=lightgray] (Max) circle [radius=0.18] ; 
\draw [line width=0.07cm, fill=white] (N11) circle [radius=0.18] ; \draw [line width=0.07cm, fill=white] (N12) circle [radius=0.18] ;
\draw [line width=0.07cm, fill=white] (N13) circle [radius=0.18] ; \draw [line width=0.07cm, fill=white] (N14) circle [radius=0.18] ;
\draw [line width=0.07cm, fill=white] (N15) circle [radius=0.18] ;
\draw [line width=0.07cm, fill=white] (N21) circle [radius=0.18] ; \draw [line width=0.07cm, fill=white] (N22) circle [radius=0.18] ;
\draw [line width=0.07cm, fill=white] (N23) circle [radius=0.18] ; \draw [line width=0.07cm, fill=white] (N24) circle [radius=0.18] ;
\draw [line width=0.07cm, fill=white] (N25) circle [radius=0.18] ;
\draw [line width=0.07cm, fill=white] (Nt1) circle [radius=0.18] ; \draw [line width=0.07cm, fill=white] (Nt2) circle [radius=0.18] ;
\draw [line width=0.07cm, fill=white] (Nt3) circle [radius=0.18] ; \draw [line width=0.07cm, fill=white] (Nt4) circle [radius=0.18] ;
\draw [line width=0.07cm, fill=white] (Nt5) circle [radius=0.18] ;
\draw [line width=0.07cm, fill=white] (Ntt1) circle [radius=0.18] ; \draw [line width=0.07cm, fill=white] (Ntt2) circle [radius=0.18] ;
\draw [line width=0.07cm, fill=white] (Ntt3) circle [radius=0.18] ; \draw [line width=0.07cm, fill=white] (Ntt4) circle [radius=0.18] ;
\draw [line width=0.07cm, fill=white] (Ntt5) circle [radius=0.18] ;

\node at (4.5,-3.1) {\Large$\widehat{0}$}; \node at (4.5,10.1) {\Large$\widehat{1}$}; 
\node at (0.7,0) {\Large $p_{1,1}$}; \node at (0.7,1.5) {\Large $p_{1,2}$}; \node at (0.7,3) {\Large $p_{1,3}$}; 
\node at (0.8,7) {\Large $p_{1,r_1}$}; 
\node at (3.7,0) {\Large $p_{2,1}$}; \node at (3.7,1.5) {\Large $p_{2,2}$}; \node at (3.7,3) {\Large $p_{2,3}$}; 
\node at (3.8,7) {\Large $p_{2,r_2}$}; 
\node at (5.2,0) {\Large $p_{t-1,1}$}; \node at (5.2,1.5) {\Large $p_{t-1,2}$}; \node at (5.2,3) {\Large $p_{t-1,3}$}; 
\node at (5.1,7) {\Large $p_{t-1,r_{t-1}}$}; 

\node at (8.4,0) {\Large $p_{t,1}$}; \node at (8.4,1.5) {\Large $p_{t,2}$}; \node at (8.4,3) {\Large $p_{t,3}$}; 
\node at (8.35,7) {\Large $p_{t,r_t}$}; 

\end{tikzpicture}
} }
\end{center}
\caption{The Hasse diagram of the poset giving the Segre product of polynomial rings, and a spanning tree (red edges)}
\label{segre_poset}
\end{figure}

Then, we can see that 
\begin{align}\label{segre_ineq}
\calC(P)=\{(z_1,\cdots,z_{t-1}) \in \RR^{t-1} \mid &-r_t \leq z_i \leq r_i \text{ for }1 \leq i \leq t-1, \\
& -r_j \leq z_i-z_j \leq r_i \text{ for }1 \leq i < j \leq t-1\}. \nonumber
\end{align}
In fact, since every cycle of $\widehat{P}$ is of the form 
$$C_{k,\ell}=(p_{k,1},p_{k,2},\cdots,p_{k,r_k+1},p_{\ell,r_\ell},\cdots,p_{\ell,1},p_{\ell,0})$$ 
where $1 \leq k < \ell \leq t$, one has 
\begin{align*}
&Z_{C_{k,\ell}}^+=\{e_{k,1}\}, \quad Z_{C_{k,\ell}}^-=\begin{cases}\{e_{\ell,1}\}, \;\; &\text{if }\ell \leq t-1, \\ \emptyset &\text{if }\ell=t \end{cases} 
\quad\text{ and }\\
&|X_{C_{k,\ell}}^+|=r_k+1 \text{ and }|X_{C_{k,\ell}}^-|=r_\ell+1\text{ for }1 \leq k < \ell \leq t. 
\end{align*}
Hence, we obtain the inequalites given in \eqref{segre_ineq}. 

Moreover, from the above description, the full-dimensional cell in $\RR^{\sum_{i=1}^tr_i+1}$ corresponding to each point $(m_1,\cdots,m_{t-1}) \in \calC(P) \cap \ZZ^{t-1}$ is given by 
\begin{align}\label{segre_cell}
\bigg\{(y',y_{i,j})_{\small\substack{1 \leq i \leq t \\ 1 \leq j \leq r_i}} \in \RR^{\sum_{i=1}^tr_i+1} \mid & -1 < y' \leq 0, \; -1 < y_{i,j} \leq 0, \\
&m_i-1 < y'+\sum_{j=1}^{r_t}y_{t,j}-\sum_{j'=1}^{r_i}y_{i,j'} \leq m_i \text{ for }1 \leq i \leq t-1\bigg\}.\nonumber\end{align}
\end{example}

\section{Non-commutative crepant resolutions of Segre products of polynomial rings}
\label{sec_NCCR}

In this section, we assume that $\kk$ is an algebraically closed field of characteristic zero. 
First, we recall the definition of non-commutative (crepant) resolutions \cite{VdB3,DITV}. 

\begin{definition}
\label{def_NCCR}
Let $A$ be a CM normal domain, and $M$ be a non-zero reflexive $A$-module. Let $E\coloneqq\End_A(M)$. 
\begin{enumerate}[(1)]
\setlength{\parskip}{0pt} 
\setlength{\itemsep}{3pt}
\item We say that $E$ is a \emph{non-commutative resolution} (= \emph{NCR}) of $A$ or $M$ \emph{gives an NCR} of $A$ if $\gldim E<\infty$. 
\item We say that $E$ is a \emph{non-commutative crepant resolution} (= \emph{NCCR}) of $A$ 
or $M$ \emph{gives an NCCR} of $A$ if $\gldim E_\mathfrak{p}=\dim A_\mathfrak{p}$ for all $\mathfrak{p}\in\Spec A$ and $E$ is an MCM $A$-module.
\end{enumerate}
In addition, we say that an NC(C)R $E=\End_A(M)$ is \emph{splitting} if $M$ is a finite direct sum of rank one reflexive $A$-modules after \cite{IN}. 
\end{definition}

\begin{remark}
\label{rem_NCCR}
\begin{enumerate}[(1)]
\setlength{\parskip}{0pt} 
\setlength{\itemsep}{3pt}
\item When $A$ is a Gorenstein normal domain, we can relax Definition~\ref{def_NCCR}(2). 
That is, $E$ is an NCCR of $A$  if and only if $\gldim E<\infty$ and $E$ is an MCM $A$-module (see \cite[Lemma~2.23]{IW1}). 
\item A splitting NCCR is also called ``toric" NCCR when $A$ is a toric ring (see \cite{Boc}).
\end{enumerate}
\end{remark}
 
As we saw in Theorem~\ref{motivation_thm}(2), any toric ring admits an NCR, 
which is given by conic divisorial ideals. 
Thus, we immediately have the following. 

\begin{corollary}
Let $R$ be a Hibi ring. Let $M_\calC$ be a direct sum of all conic divisorial ideals, that is, $M_\calC$ is a direct sum of divisorial ideals 
satisfying the condition given in Theorem~\ref{thm:conic}. 
Then,  $\End_R(M_\calC)$ is an NCR of $R$. 
\end{corollary}

We note that $M_\calC$ rarely gives an NCCR, because it is too big to allow $\End_R(M_\calC)$ to be an MCM $R$-module. 
Thus, we will consider to leave out some summands in $M_\calC$, and construct an NCCR. 
In fact, an NCCR of a toric ring given via a ``quasi-symmetric" representation (we will define this later) is constructed using a part of conic divisorial ideals (see \cite[Theorem~1.19]{SpVdB}). 
Our result in the previous section enables us to construct NCCRs for some Hibi rings, even if those do not satisfy the quasi-symmetric condition. 

Let $R$ be a Hibi ring whose class group is $\Cl(R)\cong \ZZ^{n-d}$. 
In what follows, we construct NCCRs of some Hibi rings using methods in \cite{SpVdB} heavily. 
To clarify the difference between our settings and those in \cite{SpVdB}, we rewrite $R$ as the ring of invariants 
under the action of $G\coloneqq\Hom(\Cl(R),\kk^\times)\cong(\kk^\times)^{n-d}$ on $S$, 
where $S\coloneqq \kk[x_1,\cdots,x_n]$ that gives the embedding $R\hookrightarrow S$ via $\sigma$ in (\ref{cl_seq}).
Let $\rmX(G)$ be the character group of $G$. In particular, we see that $\Cl(R)\cong\rmX(G)$. 
(Using this identification, we will use the same symbol for both of a character and the corresponding weight.) 
When we consider the prime divisor $\calD_i$ on $\Spec R$ as the element in $\rmX(G)$ via the surjection in (\ref{cl_seq}), we denote it by $\beta_i$. 
Let $V_\chi$ be the irreducible representation corresponding to a character $\chi\in\rmX(G)$, and let $W=\oplus_iV_{\beta_i}$ 
(we sometimes identify $\chi$ with $V_\chi$). 
Then, the symmetric algebra $S(W)$ of the $G$-representation $W$ is isomorphic to $S$, and this induces the action of the algebraic torus $G$ on $S$. 
That is, $g\in G$ acts on $x_i$ as $g\cdot x_i=\beta_i(g)x_i$. 
This action is generic (see \cite[Definition~1.6]{SpVdB}), and gives the $\Cl(R)$-grading on $S$, and the degree zero part 
coincides with the $G$-invariant components. In particular, we have that $R=S^G$ (see e.g., \cite[Theorem~2.1]{BG1}). 
Also, we say that $W$ is \emph{quasi-symmetric} if for every line $\ell\subset\rmX(G)_\RR$ passing through the origin, 
we have $\sum_{\beta_i\in\ell}\beta_i=0$. If $W$ is quasi-symmetric, then $R=S^G$ is Gorenstein. 

We then rewrite conic divisorial ideals in terms of modules of covariants (cf. \cite[Section~10.6]{SpVdB}). 
For a character $\chi\in\rmX(G)$, we call an $R$-module with the form $M_{\chi}\coloneqq(S\otimes_\kk V_{\chi})^G$ \emph{module of covariants}, 
which is generated by $f\in S$ with $g\cdot f=\chi(g)f$ for any $g\in G$. 
In particular, for $\chi=\sum_ia_i\calD_i\in\rmX(G)$ we see that $T(a_1,\cdots, a_n)=M_{-\chi}$ 
(we sometimes denote $T(a_1,\cdots, a_n)$ by $T(\chi)$).
As we mentioned in Lemma~\ref{conic_characterization}, a conic divisorial ideal corresponds to a divisor $\sum_i\delta_i\calD_i$ with $(\delta_i)_i\in(-1,0]^n$ in $\Cl(R)_\RR$, and this corresponds to the character $\sum_i\delta_i\beta_i\in\rmX(G)_\RR$. 
We say that a character $\chi\in\rmX(G)$ is \emph{strongly critical} with respect to $\beta_1,\cdots,\beta_n$ if 
$\chi=\sum_ia_i\beta_i$ satisfies $a_i\in(-1,0]$ for all $i$ in $\rmX(G)_\RR$. 
Therefore, we see that a conic divisorial ideal is precisely a module of covariants associated with a strongly critical character. 

\medskip

Following \cite[Section~10]{SpVdB}, we also introduce several notations. 
Let $\calA=\mc(G,S)$ be the category of finitely generated $(G,S)$-module (i.e., $G$-equivariant $S$-modules). 
We define the $(G,S)$-module $P_\chi\coloneqq V_\chi\otimes_\kk S$. 
Any projective generator of $\calA$ is given by $P_\chi$ with $\chi\in\rmX(G)$. 
For a finite subset $\calL$ of $\rmX(G)$, we set 
$$
P_\calL\coloneqq\bigoplus_{\chi\in\calL}P_\chi \text{\quad and \quad}\Lambda_\calL\coloneqq\End_\calA(P_\calL). 
$$
In addition, for $\chi\in\rmX(G)$ we set $P_{\calL,\chi}\coloneqq\Hom_\calA(P_\calL,P_\chi)$. 
If $\chi\in\calL$, we see that $P_{\calL,\chi}$ is a right projective $\Lambda_\calL$-module by projectivization \cite[II. 2.1]{ARS}. 
Since the functor $(-)^G:\refl(G,S)\rightarrow\refl(R)$ gives equivalence (see e.g., \cite[Lemma~3.3]{SpVdB}, \cite[Section~11]{Has}), 
we see that 
$$
\Hom_\calA(P_{\chi_i},P_{\chi_j})\cong(\Hom_\kk(V_{\chi_i},V_{\chi_j})\otimes_\kk S)^G\cong\Hom_R(P_{\chi_i}^G,P_{\chi_j}^G)=\Hom_R(M_{\chi_i},M_{\chi_j}). 
$$
Thus, to consider the global dimension of $\End_R(P_\calL^G)$, we may only discuss that of $\Lambda_\calL$. 
The following two lemmas are crucial to check the finiteness of the global dimension.  

\begin{lemma}[{see \cite[Lemma~10.1]{SpVdB}}]\label{key_lem1}
 One has that $\gldim\Lambda_\calL<\infty$ if and only if $\pdim_{\Lambda_\calL}P_{\calL,\chi}<\infty$ for all $\chi\in\rmX(G)$. 
\end{lemma}

Let $\rmY(G)$ be the group of one-parameter subgroups of $G$, which is isomorphic to $\ZZ^{n-d}$. 
We say that $\chi\in\rmX(G)$ is \emph{separated} from $\calL$ by $\lambda\in\rmY(G)_\RR$ if $\langle\lambda,\chi\rangle<\langle\lambda,\mu\rangle$ 
for any $\mu\in\calL$. 

Let $K_\lambda$ be the subspace of $W$ spanned by representations corresponding to $\beta_{i_j}$ with $\langle\lambda,\beta_{i_j}\rangle>0$, 
and let $d_\lambda\coloneqq\dim_\kk K_\lambda$. 
Consider the Koszul resolution:  
$$
0\longrightarrow \wedge^{d_\lambda}K_\lambda\otimes_\kk S \longrightarrow\wedge^{d_\lambda-1}K_\lambda\otimes_\kk S \longrightarrow \cdots \longrightarrow S \longrightarrow S(W/K_\lambda)\longrightarrow 0.
$$ 
Applying $(\chi\otimes_\kk-)=(V_\chi\otimes_\kk-)$ to this, we have the exact sequence $C_{\lambda,\chi}$: 
\begin{equation}
\begin{split}
\label{chi_sequence}
0\longrightarrow (\chi\otimes_\kk\wedge^{d_\lambda}K_\lambda)\otimes_\kk S \overset{\delta_{d_\lambda}}{\longrightarrow}(\chi\otimes_\kk\wedge^{d_\lambda-1}K_\lambda)&\otimes_\kk S \overset{\delta_{d_{\lambda-1}}}{\longrightarrow} \cdots \\ &\overset{\delta_1}{\longrightarrow} \chi\otimes_\kk S \longrightarrow \chi\otimes_\kk S(W/K_\lambda)\longrightarrow 0. 
\end{split}
\end{equation}
We note that for $p=1,\cdots, d_\lambda$ the $(G,S)$-module $(\chi\otimes_\kk\wedge^pK_\lambda)\otimes_\kk S$ is decomposed as the direct sum of $(G,S)$-modules $P_\mu$ 
of the form $\mu=\chi+\beta_{i_1}+\cdots+\beta_{i_p}$ where $\{i_1, \cdots, i_p\}\subset \{1, \cdots, n\}$, 
$i_j\neq i_{j^\prime}$ if $j\neq j^\prime$, and $\langle\lambda,\beta_{i_j}\rangle>0$. 
Let $C_{\calL,\lambda,\chi}$ be the complex obtained by applying $\Hom_\calA(P_\calL,-)$ to $C_{\lambda,\chi}$. 
Then the following lemma holds. 
We remark that $\Hom_\calA(P_\calL,\chi\otimes_\kk S(W/K_\lambda))=0$ in this situation, and these kinds of arguments also hold 
for any toric rings which are not necessarily Hibi rings. 

\begin{lemma}[{see \cite[Lemma~10.2]{SpVdB}}]\label{key_lem2}
Suppose that $\chi\in\rmX(G)$ is separated from $\calL$ by $\lambda\in\rmY(G)_\RR$. 
Then, the complex $C_{\calL,\lambda,\chi}$ is acyclic. 
In addition, the $0$-th term of $C_{\calL,\lambda,\chi}$ is $P_{\calL,\chi}$ and for $p=1,\cdots, d_\lambda$ 
the $-p$-th term is the direct sum of $P_{\calL,\mu}$ of the form 
$$
\mu=\chi+\beta_{i_1}+\cdots+\beta_{i_p}
$$
where $\{i_1, \cdots, i_p\}\subset \{1, \cdots, n\}$, $i_j\neq i_{j^\prime}$ if $j\neq j^\prime$, and $\langle\lambda,\beta_{i_j}\rangle>0$. 
\end{lemma}

\subsection{Construction of NCCRs for Segre products of polynomial rings} 

In this subsection, we construct an NCCR for the Segre products of polynomial rings, which are realized as Hibi rings. 
First, we remark that a ring admitting an NCCR is $\QQ$-Gorenstein (see \cite[Theorem~1.1]{DITW}). 
Since the class group of a Hibi ring is a free abelian group, 
a Hibi ring admitting an NCCR is Gorenstein. We recall that Gorenstein Hibi rings arise from pure posets (\cite[Section~3]{Hibi}). 
For the case of the Segre product of polynomial rings, we see that it is Gorenstein if and only if polynomial rings have the same variables. 

Let $R=S_1\# S_2\#\cdots\# S_t$ be the Segre product of $r$-dimensional polynomial rings, where $r\ge2$. 
This is realized as the Hibi ring associated with the pure poset shown in Figure~\ref{segre_poset_gor}, 
and we have that $\Cl(R)\cong\ZZ^{t-1}$ (see Example \ref{segre}). 
Let $G=\Hom(\Cl(R),\kk^\times)$. Then, we see that $R=S^G$ where $S=S(W)$ is a polynomial ring with $tr$ variables 
arising from the representation $W=\bigoplus_{s=1}^{tr}V_{\beta_s}$. Here, $\beta_s\in\rmX(G)\cong\ZZ^{t-1}$ is the character corresponding to 
a prime divisor $\calD_s$ on $\Spec R$. We see that by the relations (\ref{relation_div_hibi}) each $\beta_s$ takes the form 
$\bar{\beta}_i=(0,\cdots,0,\overset{i}{\check{1}},0,\cdots,0) \in \ZZ^{t-1}$ for $i=1,\cdots,t-1$ or $\bar{\beta}_t=(-1,\cdots,-1)$, 
and each of these has the multiplicity $r$. 

\begin{figure}[h]
\begin{center}
{\scalebox{0.6}{
\begin{tikzpicture}

\coordinate (Min) at (3,-1.5); \coordinate (Max) at (3,6.5);
\coordinate (N11) at (0,0); \coordinate (N12) at (0,1); \coordinate (N13) at (0,2); \coordinate (N14) at (0,4); \coordinate (N15) at (0,5);
\coordinate (N21) at (2,0); \coordinate (N22) at (2,1); \coordinate (N23) at (2,2); \coordinate (N24) at (2,4); \coordinate (N25) at (2,5);
\coordinate (Ntt1) at (4,0); \coordinate (Ntt2) at (4,1); \coordinate (Ntt3) at (4,2); \coordinate (Ntt4) at (4,4); \coordinate (Ntt5) at (4,5);
\coordinate (Nt1) at (6,0); \coordinate (Nt2) at (6,1); \coordinate (Nt3) at (6,2); \coordinate (Nt4) at (6,4); \coordinate (Nt5) at (6,5);
\draw[line width=0.05cm]  (N11)--(N12); \draw[line width=0.05cm]  (N12)--(N13); \draw[line width=0.05cm]  (N13)--(0,2.5); 
\draw[line width=0.05cm]  (0,3.5)--(N14); \draw[line width=0.05cm]  (N14)--(N15); \draw[line width=0.05cm,dotted]  (0,2.8)--(0,3.2);
\draw[line width=0.05cm]  (N21)--(N22); \draw[line width=0.05cm]  (N22)--(N23); \draw[line width=0.05cm]  (N23)--(2,2.5); 
\draw[line width=0.05cm]  (2,3.5)--(N24); \draw[line width=0.05cm]  (N24)--(N25); \draw[line width=0.05cm,dotted]  (2,2.8)--(2,3.2);
\draw[line width=0.05cm]  (Nt1)--(Nt2); \draw[line width=0.05cm]  (Nt2)--(Nt3); \draw[line width=0.05cm]  (Nt3)--(6,2.5); 
\draw[line width=0.05cm]  (6,3.5)--(Nt4); \draw[line width=0.05cm]  (Nt4)--(Nt5); \draw[line width=0.05cm,dotted]  (6,2.8)--(6,3.2);
\draw[line width=0.05cm]  (Ntt1)--(Ntt2); \draw[line width=0.05cm]  (Ntt2)--(Ntt3); \draw[line width=0.05cm]  (Ntt3)--(4,2.5); 
\draw[line width=0.05cm]  (4,3.5)--(Ntt4); \draw[line width=0.05cm]  (Ntt4)--(Ntt5); \draw[line width=0.05cm,dotted]  (4,2.8)--(4,3.2);
\draw[line width=0.05cm]  (N11)--(Min); \draw[line width=0.05cm]  (N21)--(Min); \draw[line width=0.05cm]  (Ntt1)--(Min); \draw[line width=0.05cm]  (Nt1)--(Min); 
\draw[line width=0.05cm]  (N15)--(Max); \draw[line width=0.05cm]  (N25)--(Max); \draw[line width=0.05cm]  (Ntt5)--(Max); \draw[line width=0.05cm]  (Nt5)--(Max); 
\draw[line width=0.05cm,dotted]  (2.5,3)--(3.5,3); 

\draw [line width=0.05cm, fill=lightgray] (Min) circle [radius=0.15] ; \draw [line width=0.05cm, fill=lightgray] (Max) circle [radius=0.15] ; 
\draw [line width=0.05cm, fill=white] (N11) circle [radius=0.15] ; \draw [line width=0.05cm, fill=white] (N12) circle [radius=0.15] ;
\draw [line width=0.05cm, fill=white] (N13) circle [radius=0.15] ; \draw [line width=0.05cm, fill=white] (N14) circle [radius=0.15] ;
\draw [line width=0.05cm, fill=white] (N15) circle [radius=0.15] ;
\draw [line width=0.05cm, fill=white] (N21) circle [radius=0.15] ; \draw [line width=0.05cm, fill=white] (N22) circle [radius=0.15] ;
\draw [line width=0.05cm, fill=white] (N23) circle [radius=0.15] ; \draw [line width=0.05cm, fill=white] (N24) circle [radius=0.15] ;
\draw [line width=0.05cm, fill=white] (N25) circle [radius=0.15] ;
\draw [line width=0.05cm, fill=white] (Nt1) circle [radius=0.15] ; \draw [line width=0.05cm, fill=white] (Nt2) circle [radius=0.15] ;
\draw [line width=0.05cm, fill=white] (Nt3) circle [radius=0.15] ; \draw [line width=0.05cm, fill=white] (Nt4) circle [radius=0.15] ;
\draw [line width=0.05cm, fill=white] (Nt5) circle [radius=0.15] ;
\draw [line width=0.05cm, fill=white] (Ntt1) circle [radius=0.15] ; \draw [line width=0.05cm, fill=white] (Ntt2) circle [radius=0.15] ;
\draw [line width=0.05cm, fill=white] (Ntt3) circle [radius=0.15] ; \draw [line width=0.05cm, fill=white] (Ntt4) circle [radius=0.15] ;
\draw [line width=0.05cm, fill=white] (Ntt5) circle [radius=0.15] ;

\node at (3,-2.1) {\Large$\widehat{0}$}; \node at (3,7.1) {\Large$\widehat{1}$}; 

\node at (9,0) {$$}; 
\draw [line width=0.03cm, decorate,decoration={brace,amplitude=10pt}](-1.1,0) -- (-1.1,5) node[black,midway,xshift=-1.3cm,yshift=0.2cm] {\Large $r-1$}
node[black,midway,xshift=-1.1cm,yshift=-0.2cm] {\Large vertices}; 
\draw [line width=0.03cm, decorate,decoration={brace,amplitude=10pt,mirror}](0,-2.6) -- (6,-2.6) node[black,midway,yshift=-0.7cm] {\Large $t$ columns}; 

\end{tikzpicture}
} }
\end{center}
\caption{The Hasse diagram of the poset giving the Segre product of polynomial rings that is Gorenstein}
\label{segre_poset_gor}
\end{figure}

We see that by Theorem~\ref{thm:conic} (see also Example~\ref{segre}) conic divisorial ideals of $R$ are represented by
\begin{align*}
\calC(R)=\{(c_1,\cdots,c_{t-1})\in\Cl(R) \mid & -(r-1) \leq c_i \leq r-1 \ \text{for} \  1\leq i\leq t-1, \\
& -(r-1) \leq c_i-c_j \leq r-1 \ \text{for} \  1\leq i<j\leq t-1\}
\end{align*}
as elements in $\Cl(R)\cong\rmX(G)$. 
We remark that as we mentioned a conic divisorial ideal corresponds to a strongly critical character. 
Thus, the description of $\calC(R)$ is also deduced from the description of the characters $\bar{\beta}_1,\cdots, \bar{\beta}_t$, which have the multiplicity $r$. 
Also, we will consider elements in $\calC(R)$ as strongly critical characters. We then set 
\begin{equation}
\label{def_L}
\calL\coloneqq\{c=(c_1,\cdots,c_{t-1}) \in \calC(R) \mid 0\le c_i\le r-1 \text{\ \ for any \ } i=1,\cdots,t-1\} 
\end{equation}
and this finite set gives an NCCR of $R$ as follows. 

\begin{theorem}
\label{NCCR_Segre}
Let $R$ be the Segre product of $r$-dimensional polynomial rings as above. 
Let 
$$
M_\calL\coloneqq\bigoplus_{\chi\in\calL}M_\chi=\bigoplus_{\chi\in\calL}T(-\chi), 
$$
where $\calL$ is a finite set $(\ref{def_L})$. Then, $\End_R(M_\calL)$ is an NCCR of $R$. 
$($Remark that this NCCR is splitting and the number of non-isomorphic direct summands of $M_\calL$ is $r^{t-1}$.$)$  
\end{theorem}

\begin{remark}
If $r=t=2$, this $R$ is isomorphic to the conifold $\kk[X,Y,Z,W]/(XW-YZ)$ (or $3$-dimensional $A_1$-singularity) and an NCCR of the conifold is one of the prototypes for the study of NCCRs (see e.g., \cite{VdB3}), and Theorem~\ref{NCCR_Segre} is a generalization of this well-known example. 
Also, we easily see that $R$ is coming from a quasi-symmetric representation if and only if $t=2$, 
in which case the existence of NCCRs is covered by \cite{SpVdB}, and is also studied in \cite{Kuz} using tilting bundles. 
Furthermore, if $t=2$ we can consider R as a determinantal ring, and NCCRs of determinantal rings had been constructed in \cite{BLVdB}. 

The existence of NCCRs for other Gorenstein Hibi rings is still open. 
Some affirmative results can be found in \cite{Nak2}, that is, 
the second author constructed NCCRs for a Gorenstein Hibi ring whose class group is $\ZZ$ or $\ZZ^2$. 
\end{remark}

\begin{proof}[Proof of Theorem~\ref{NCCR_Segre}]
We set 
$$
\Sigma\coloneqq\left\{ \sum_{s=1}^{tr} a_s\beta_s \mid a_s\in(-1,0] \right\}\subset\rmX(G)_\RR. 
$$
That is, an element in $\Sigma$ is a strongly critical character. 
Thus, identifying elements in $\Sigma$ with certain elements in $\rmX(G)$, we see that $\rmX(G)\cap\Sigma$ coincides with $\calC(R)$. 
Let 
$$
\widetilde{\calL}\coloneqq\{c=(c_1,\cdots,c_{t-1}) \in \calC(R) \mid  -(r-1) \leq c_i \leq r-1  \text{\ \ for any \ } i=1,\cdots,t-1\}. 
$$
We note that $\calL \subset \calC(R)=\rmX(G)\cap\Sigma \subset \widetilde{\calL}$. 
Furthermore, we set 
$$
\widetilde{\calL}_j\coloneqq\{c=(c_1,\cdots,c_{t-1}) \in \widetilde{\calL} \mid c_{j+1} \geq 0, \cdots, c_{t-1} \geq 0\} 
$$
for $0 \leq j \leq t-1$. 
Note that $\calL=\widetilde{\calL}_0 \subset \widetilde{\calL}_1 \subset \cdots \subset \widetilde{\calL}_{t-1} =\widetilde{\calL}$. 
For $1 \leq k \leq r-1$, let 
$$
\widetilde{\calL}_j(k)\coloneqq\{c=(c_1,\cdots,c_{t-1}) \in \widetilde{\calL}_j \mid c_j \geq -k \}.
$$

\medskip

In the following, we show the finiteness of the global dimension of $\End_R(M_\calL)$. 
To show this, we may only consider $\Lambda_\calL=\End_\calA(\bigoplus_{\chi\in\calL}P_\chi)$, 
and check that $\pdim_{\Lambda_\calL}P_{\calL,\chi}<\infty$ for any $\chi\in\rmX(G)$ by Lemma~\ref{key_lem1}. 
We first show that $\pdim_{\Lambda_\calL}P_{\calL,\chi}<\infty$ for any $\chi \in \widetilde{\calL}$. 
We prove this assertion by the induction on $j$ and $k$. (We remark that $\widetilde{\calL}=\widetilde{\calL}_{t-1}(r-1)$.) 

For any $\chi \in \widetilde{\calL}_j \setminus \widetilde{\calL}_{j-1}$, 
let $\lambda=(0,\cdots,0,\overset{j}{\check{1}},0,\cdots,0)$. Then, we can see that 
$$\langle \lambda, \chi \rangle < 0 \leq \langle \lambda, \chi^\prime\rangle \;\; \text{for any }\chi^\prime\in \calL.$$ 
Hence, $\chi$ is separated from $\calL$ by $\lambda$, 
and we have that $\langle \lambda, \bar{\beta}_j  \rangle > 0$ and $\langle \lambda, \bar{\beta}_{i} \rangle \leq 0$ for any $i\not=j$. 

\medskip

\noindent({\bf The case $j=1$}): 
For any $\chi \in \widetilde{\calL}_1(1) \setminus \calL$, we see that 
$\chi + \beta_{i_1} + \cdots + \beta_{i_p} \in \calL$, where $\beta_{i_1}=\cdots=\beta_{i_p}=\bar{\beta}_1$ and $1 \leq p \leq r$. 
Hence, by Lemma~\ref{key_lem2} we have that $\pdim_{\Lambda_\calL}P_{\calL,\chi}<\infty$. 
We then assume that $\pdim_{\Lambda_\calL}P_{\calL,\chi}<\infty$ for any $\chi \in \widetilde{\calL}_1(k)$. 
Then, for any $\chi' \in \widetilde{\calL}_1(k+1) \setminus \widetilde{\calL}_1(k)$, we see that 
$\chi^\prime + \beta_{i_1} + \cdots + \beta_{i_p} \in \widetilde{\calL}_1(k)$. 
Hence, $\pdim_{\Lambda_\calL}P_{\calL,\chi^\prime}<\infty$ by Lemma~\ref{key_lem2}. 
Therefore, $\pdim_{\Lambda_\calL}P_{\calL,\chi}<\infty$ for any $\chi \in \widetilde{\calL}_1$. 

\noindent({\bf The case $j>1$}): 
We assume that $\pdim_{\Lambda_\calL}P_{\calL,\chi}<\infty$ for any $\chi \in \widetilde{\calL}_{j-1}$ with $j \geq 2$. 
For any $\chi \in \widetilde{\calL}_j(1) \setminus \widetilde{\calL}_{j-1}$, we see that 
$\chi + \beta_{i_1} + \cdots + \beta_{i_p} \in \widetilde{\calL}_{j-1}$, where $\beta_{i_1}=\cdots=\beta_{i_p}=\bar{\beta}_j$ and $1 \leq p \leq r$. 
Hence, by Lemma~\ref{key_lem2} we have that $\pdim_{\Lambda_\calL}P_{\calL,\chi}<\infty$. 
We then assume that $\pdim_{\Lambda_\calL}P_{\calL,\chi}<\infty$ for any $\chi \in \widetilde{\calL}_j(k)$. 
Then, for any $\chi' \in \widetilde{\calL}_j(k+1) \setminus \widetilde{\calL}_j(k)$, we see that 
$\chi^\prime + \beta_{i_1} + \cdots + \beta_{i_p} \in \widetilde{\calL}_1(k)$. 
Hence, $\pdim_{\Lambda_\calL}P_{\calL,\chi^\prime}<\infty$ by Lemma~\ref{key_lem2}. 
Therefore, $\pdim_{\Lambda_\calL}P_{\calL,\chi}<\infty$ for any $\chi \in \widetilde{\calL}_j$. 

\medskip

Consequently, we obtain that $\pdim_{\Lambda_\calL}P_{\calL,\chi}<\infty$ for any $\chi \in \widetilde{\calL}_{t-1}(r-1)=\widetilde{\calL}$, 
and so is for any $\chi\in\rmX(G)\cap\Sigma$ in particular. 
If we assume that $\gldim\Lambda_\calL=\infty$, then by Lemma~\ref{key_lem1} there exists $\chi\in\rmX(G)$ such that 
$\pdim_{\Lambda_\calL}P_{\calL,\chi}=\infty$, 
and this $\chi$ should be in $\rmX(G)\setminus\widetilde{\calL}\subset\rmX(G)\setminus\rmX(G)\cap\Sigma$ by the above observation. 
Using the same argument as in \cite[Subsection~10.3]{SpVdB}, we can conclude that this is a contradiction. 

In order to show that $\End_R(M_\calL)$ is an NCCR, we have to show that $\End_R(M_\calL)$ is an MCM $R$-module (see Remark~\ref{rem_NCCR}). 
This follows from Lemma~\ref{MCM_lem1} and \ref{MCM_lem2} below. 
\end{proof}

To complete the proof of Theorem~\ref{NCCR_Segre}, we prove the following Lemmas. 

\begin{lemma}
\label{MCM_lem1}
For any $\chi,\chi' \in \calL$, we have that $\chi-\chi'\in \widetilde{\calL}$. 
\end{lemma}

\begin{proof}
For $a >0$, it is enough to show that $\{\alpha-\alpha' \mid \alpha,\alpha' \in [0,a]^d\} \subset [-a,a]^d$ and this is obvious.  
\end{proof}

Before moving to another lemma, we note that using $\bar{\beta}_1,\cdots,\bar{\beta_t}$ each character $\chi\in\rmX(G)$ is described as 
$\chi=\sum_{i=1}^tc_i\bar{\beta}_i$. 
It is known that a module of covariants $M_\chi$ is a rank one MCM $R$-module 
if and only if for some permutation $i_1,\cdots,i_t$ of $1,\cdots,t$, $0 \leq c_{i_{j+1}} - c_{i_j} \leq r-1$ holds for $1 \leq j \leq t-1$ (see \cite[Section~2]{Bru}). 
Here, $(c_1,\cdots,c_t) \in \ZZ^t$ can be identified with 
$(c_1-c_t,\cdots,c_{t-1}-c_t) \in \Cl(R)$ via the relation $\bar{\beta}_t=-(\bar{\beta}_1+\cdots+\bar{\beta}_{t-1})$. 

\begin{lemma}
\label{MCM_lem2}
For any $\chi \in \widetilde{\calL}$, $M_\chi$ is a rank one MCM $R$-module. 
\end{lemma}

\begin{proof}
We may write $\chi=\sum_{i=1}^tc_i\bar{\beta}_i$ as $\chi=(x_1,\cdots,x_{t-1})$ 
with $x_i-x_j=c_i-c_j$ for any $1 \leq i \not= j \leq t-1$, and $x_i = c_i-c_t$ for any $1 \leq i \leq t-1$. 
For any $\chi\in\widetilde{\calL}$, we have that $-(r-1) \leq x_i \leq r-1$ for all $i$, 
and there exist a certain permutation $i_1,\cdots,i_{t-1}$ of $1,\cdots,t-1$ and $0 \leq \ell \leq t-1$ such that 
$$-(r-1) \leq x_{i_1} \leq x_{i_2} \leq \cdots \leq x_{i_\ell} \leq 0 \leq x_{i_{\ell+1}} \leq \cdots \leq x_{i_{t-1}} \leq r-1.$$ 
Consider the sequence $i_1,\cdots,i_\ell,t,i_{\ell+1},\cdots,i_{t-1}$ of $1,\cdots,t$. 
\begin{itemize}
\setlength{\parskip}{0pt} 
\setlength{\itemsep}{3pt}
\item For each $1 \leq j \leq \ell-1$, since $0 \leq x_{i_{j+1}}-x_{i_j} \leq -x_{i_j} \leq r-1$ 
and $x_{i_{j+1}}-x_{i_j}=c_{i_{j+1}}-c_{i_j}$, we have $0 \leq c_{i_{j+1}}-c_{i_j} \leq r-1$. 
\item For each $\ell+1 \leq j \leq t-1$, since $0 \leq x_{i_{j+1}}-x_{i_j} \leq x_{i_{j+1}} \leq r-1$ 
and $x_{i_{j+1}}-x_{i_j}=c_{i_{j+1}}-c_{i_j}$, we have $0 \leq c_{i_{j+1}}-c_{i_j} \leq r-1$. 
\item We see that $0 \leq -x_{i_\ell}=c_t-c_{i_\ell} \leq r-1$ and $0 \leq x_{i_{\ell+1}}=c_{i_{\ell+1}}-c_t \leq r-1$. 
\end{itemize}
Therefore, $\chi=\sum_{i=1}^t c_i\bar{\beta}_i$ satisfies the condition that $(c_1,\cdots,c_t)$ corresponds to a rank one MCM $R$-module. 
\end{proof}

In this manner, we can obtain an NCCR $\End_R(M_\calL)$ of Segre products of polynomial rings. 
Using this module $M_\calL$, we also obtain other modules giving NCCRs as follows. 
In the next section, we will give essentially different NCCRs using ``mutation". 

\begin{corollary}
Let the notation be the same as Theorem~\ref{NCCR_Segre}. 
For each divisorial ideal $I$ of $R$, $(M_\calL\otimes_RI)^{**}$ gives an NCCR of $R$. 
\end{corollary}

\begin{proof}
This follows from the fact $\End_R((M_\calL\otimes_RI)^{**})\cong\End_R(M_\calL)$. 
\end{proof}

\begin{example}
Suppose that $R$ is the Segre product of polynomial rings with $r=3, t=3$, 
in which case we have that $\Cl(R)\cong\ZZ^2$. 
In the following figure, each dot represents a divisorial ideal that is an MCM $R$-module, 
and the ones contained in the gray region correspond to conic classes in particular. 
Furthermore, the red ones correspond to elements in $\calL$, especially the red $\times$ represents the origin $(0,0)$. 
Therefore, these gives an NCCR of $R$. 

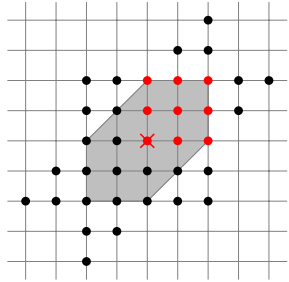
\begin{center}
{\scalebox{0.38}{
\begin{tikzpicture}
\coordinate (00) at (0,0); \coordinate (10) at (1,0); \coordinate (01) at (0,1); \coordinate (-10) at (-1,0); \coordinate (0-1) at (0,-1);  
\coordinate (11) at (1,1); \coordinate (-11) at (-1,1); \coordinate (-1-1) at (-1,-1); \coordinate (1-1) at (1,-1);   
\coordinate (20) at (2,0); \coordinate (21) at (2,1); \coordinate (22) at (2,2); \coordinate (02) at (0,2); 
\coordinate (12) at (1,2); \coordinate (-20) at (-2,0); \coordinate (-2-1) at (-2,-1); \coordinate (-2-2) at (-2,-2); 
\coordinate (0-2) at (0,-2); \coordinate (-1-2) at (-1,-2); 
\coordinate (31) at (3,1); \coordinate (32) at (3,2); \coordinate (42) at (4,2);   
\coordinate (13) at (1,3); \coordinate (23) at (2,3); \coordinate (24) at (2,4); 
\coordinate (2-1) at (2,-1); \coordinate (2-2) at (2,-2); \coordinate (1-2) at (1,-2); 
\coordinate (-3-1) at (-3,-1); \coordinate (-3-2) at (-3,-2); \coordinate (-4-2) at (-4,-2);   
\coordinate (-1-3) at (-1,-3); \coordinate (-2-3) at (-2,-3); \coordinate (-2-4) at (-2,-4); 
\coordinate (-21) at (-2,1); \coordinate (-22) at (-2,2); \coordinate (-12) at (-1,2); 

\filldraw [color=lightgray] (2,0)--(2,2)--(0,2)--(-2,0)--(-2,-2)--(0,-2)--(2,0) ; 
\draw [step=1, gray] (-4.6,-4.6) grid (4.6,4.6);
\draw [color=gray] (2,0)--(2,2)--(0,2)--(-2,0)--(-2,-2)--(0,-2)--(2,0) ; 

\draw [red,line width=0.05cm, fill=red] (00) circle [radius=0.12] ; 
\draw [red,line width=0.05cm] (-0.23,-0.23)--(0.23,0.23); \draw [red,line width=0.05cm] (0.23,-0.23)--(-0.23,0.23); 
\draw [red,line width=0.05cm, fill=red] (10) circle [radius=0.12] ; \draw [red,line width=0.05cm, fill=red] (01) circle [radius=0.12] ; 
\draw [red,line width=0.05cm, fill=red] (11) circle [radius=0.12] ; \draw [red,line width=0.05cm, fill=red] (20) circle [radius=0.12] ; 
\draw [red,line width=0.05cm, fill=red] (21) circle [radius=0.12] ; \draw [red,line width=0.05cm, fill=red] (22) circle [radius=0.12] ; 
\draw [red,line width=0.05cm, fill=red] (02) circle [radius=0.12] ; \draw [red,line width=0.05cm, fill=red] (12) circle [radius=0.12] ; 

\draw [line width=0.05cm, fill=black] (-10) circle [radius=0.12] ; \draw [line width=0.05cm, fill=black] (0-1) circle [radius=0.12] ; 
\draw [line width=0.05cm, fill=black] (-11) circle [radius=0.12] ; \draw [line width=0.05cm, fill=black] (-1-1) circle [radius=0.12] ;
\draw [line width=0.05cm, fill=black] (1-1) circle [radius=0.12] ; \draw [line width=0.05cm, fill=black] (-20) circle [radius=0.12] ; 
\draw [line width=0.05cm, fill=black] (-2-1) circle [radius=0.12] ; \draw [line width=0.05cm, fill=black] (-2-2) circle [radius=0.12] ; 
\draw [line width=0.05cm, fill=black] (0-2) circle [radius=0.12] ; \draw [line width=0.05cm, fill=black] (-1-2) circle [radius=0.12] ; 
\draw [line width=0.05cm, fill=black] (31) circle [radius=0.12] ; \draw [line width=0.05cm, fill=black] (32) circle [radius=0.12] ; 
\draw [line width=0.05cm, fill=black] (42) circle [radius=0.12] ; \draw [line width=0.05cm, fill=black] (13) circle [radius=0.12] ; 
\draw [line width=0.05cm, fill=black] (23) circle [radius=0.12] ; \draw [line width=0.05cm, fill=black] (24) circle [radius=0.12] ; 
\draw [line width=0.05cm, fill=black] (2-1) circle [radius=0.12] ; \draw [line width=0.05cm, fill=black] (2-2) circle [radius=0.12] ; 
\draw [line width=0.05cm, fill=black] (1-2) circle [radius=0.12] ; 
\draw [line width=0.05cm, fill=black] (-3-1) circle [radius=0.12] ; \draw [line width=0.05cm, fill=black] (-3-2) circle [radius=0.12] ; 
\draw [line width=0.05cm, fill=black] (-4-2) circle [radius=0.12] ; \draw [line width=0.05cm, fill=black] (-1-3) circle [radius=0.12] ; 
\draw [line width=0.05cm, fill=black] (-2-3) circle [radius=0.12] ; \draw [line width=0.05cm, fill=black] (-2-4) circle [radius=0.12] ; 
\draw [line width=0.05cm, fill=black] (-21) circle [radius=0.12] ; \draw [line width=0.05cm, fill=black] (-22) circle [radius=0.12] ; 
\draw [line width=0.05cm, fill=black] (-12) circle [radius=0.12] ; 
\end{tikzpicture} }}
\end{center}
\end{example}

\section{Mutations of NCCRs for toric rings}
\label{sec_mutation}

In the previous section, we construct an NCCR of the Segre product of polynomial rings $R=S_1\# S_2\#\cdots\# S_t$. 
In this section, we introduce the operation called \emph{mutation}, and we can obtain another module giving an NCCR from a given one via this operation. 
In particular, we will define the mutation for $(\calL,\chi,\lambda)$ where $\calL\subset\rmX(G)$ is a finite set giving an NCCR, and $\chi\in\calL$ 
is a character separated from $\calL{\setminus}\{\chi\}$ by a certain one parameter subgroup $\lambda$. 
Although, our main interest lies in the Segre product of polynomial rings, the following argument is valid for any complete local Gorenstein toric ring 
admitting NCCRs. 
Thus, $R$ denotes such a toric ring in the rest of this section. We remark that Krull-Schmidt theorem holds in our situation. 

We first consider an $R$-module $M=\bigoplus_{i\in I}M_i$ where $I=\{1, \cdots, n\}$. 
We assume that $M$ is \emph{basic}, that is, $M_i$'s are mutually non-isomorphic. 
Following \cite[Section~6]{IW1}, we define the mutation of $M$ at $i\in I$.
For each $i\in I$, we set $M_{I\setminus\{i\}}=\bigoplus_{j\in I\setminus\{i\}}M_j$. 
We say that a morphism $\varphi:N\rightarrow M_i$ is a \emph{right $(\add_R M_{I\setminus\{i\}})$-approximation} of $M_i$ if 
$N\in\add_R M_{I\setminus\{i\}}$ and 
$$\Hom_R(M_{I\setminus\{i\}}, N)\xrightarrow{\varphi\cdot}\Hom_R(M_{I\setminus\{i\}}, M_i)$$
is surjective. In addition, we say that $\varphi$ is \emph{minimal} if $\phi\in\End_R(N)$ satisfies $\varphi{\cdot}\phi=\varphi$, then $\phi$ is an automorphism, 
equivalently if non-zero direct summands of $N$ are not mapped to zero via $\varphi$. 
Since $(\add_R M_{I\setminus\{i\}})$ is contravariantly finite and $R$ is complete, 
a minimal right $(\add_R M_{I\setminus\{i\}})$-approximation $\varphi$ exists and is unique up to isomorphism. 
We then define the \emph{right mutation} $\mu^+_i$ of $M$ at $i$ as $$\mu^+_i(M)\coloneqq M_{I\setminus\{i\}}\oplus \Ker\varphi.$$ 
Also, we define the \emph{left mutation} $\mu^-_i$ of $M$ at $i\in I$ as $\mu^-_i(M)\coloneqq (\mu^+_i(M^*))^*$ where $(-)^*$ is the $R$-dual. 

Here, we collect some properties of these mutations. 

\begin{proposition}[{see \cite[Proposition~6.5, Theorem~6.8, and 6.10]{IW1}}]
\label{basic_prop_mutation}
Let the notation be the same as above. Then, we have the followings. 
\begin{enumerate}[\rm (1)]
\setlength{\parskip}{0pt} 
\setlength{\itemsep}{3pt}
\item We have that $\mu^-_i(\mu^+_i(M))=M$ and $\mu^+_i(\mu^-_i(M))=M$. 
\item If $M$ gives an NCCR of $R$, then so do $\mu^+_i(M)$ and $\mu^-_i(M)$. 
\item $\End_R(M)$, $\End_R(\mu^+_i(M))$ and $\End_R(\mu^-_i(M))$ are all derived equivalent. 
\end{enumerate}
\end{proposition}

Then, we apply the above mutation to the $R$-module $M_\calL=\bigoplus_{\eta\in\calL}M_\eta$ where $\calL$ is a finite subset of $\rmX(G)$ 
which is not necessarily equal to the one used in Theorem~\ref{NCCR_Segre}. 
In particular, by combining several mutations, we introduce the mutation of $M_\calL$ at $\chi\in\calL$ with respect to 
a certain one parameter subgroup $\lambda\in\rmY(G)_\RR$. 
To define this, the acyclic complex appearing in Lemma~\ref{key_lem2} is the main ingredient. 
We first set $P_{\calL\setminus\chi}=\bigoplus_{\eta\in\calL\setminus\{\chi\}}P_\eta$, 
and consider the complex $C_{\lambda,\chi}$ given in (\ref{chi_sequence}). 
Let $C_{\calL\setminus\chi,\lambda,\chi}$ be the complex obtained by applying $\Hom_\calA(P_{\calL\setminus\chi},-)$ to $C_{\lambda,\chi}$. 

Here, we consider a one parameter subgroup $\lambda\in\rmY(G)_\RR$ satisfying the following conditions: 
\begin{itemize}
\setlength{\parskip}{0pt} 
\setlength{\itemsep}{3pt}
\item [$(\rmY 1)$] $\lambda$ separates $\chi\in\rmX(G)$ from a finite set $\calM\subset\rmX(G)$. 
\item [$(\rmY 2)$] For all $p=1,\cdots,d_{\lambda-1}$, one has $\chi+\beta_{i_1}+\cdots+\beta_{i_p}\in\calM$ 
where $\{i_1, \cdots, i_p\}\subset \{1, \cdots, n\}$, $i_j\neq i_{j^\prime}$ if $j\neq j^\prime$, and $\langle\lambda,\beta_{i_j}\rangle>0$. 
\end{itemize}
Let $\rmY_{\calM,\chi}$ be the set of one parameter subgroups satisfying the above conditions $(\rmY 1)$ and $(\rmY 2)$. 
We remark that sometimes such a one parameter subgroup $\lambda\in\rmY_{\calM,\chi}$ does not exist for a given $\chi$. 
Using these, we have the following theorem. 

\begin{theorem}
\label{mutation_NCCR} 
Let $R$ be a complete local Gorenstein toric ring and 
$M_\calL=\bigoplus_{\eta\in\calL}M_\eta$ be a module giving a splitting NCCR of $R$. We pick a character $\chi\in\calL$. 
If there exists a one parameter subgroup $\lambda\in\rmY_{\calL\setminus\chi,\chi}$, then $M_{\calL{\setminus}\chi}\oplus M_\nu$ also gives a splitting NCCR of $R$ where 
$$\nu=\chi+\beta_{i_1}+\cdots+\beta_{i_{d_\lambda}}$$ 
with $i_j\neq i_{j^\prime}$ if $j\neq j^\prime$, and $\langle\lambda,\beta_{i_j}\rangle>0$ for all $j=1,\cdots,d_\lambda$. 
Furthermore, $\End_R(M_\calL)$ and $\End_R(M_{\calL{\setminus}\chi}\oplus M_\nu)$ are derived equivalent. 
\end{theorem}

\begin{proof}
Consider $C_{\lambda,\chi}$ given in (\ref{chi_sequence}). 
If $\lambda\in\rmY_{\calL\setminus\chi,\chi}$, then we see that $C_{\calL{\setminus}\chi,\lambda,\chi}$ is acyclic by the condition $(\rmY 1)$ and Lemma~\ref{key_lem2}. 
(Remark that the following arguments do not depend on a choice of $\lambda\in\rmY_{\calL\setminus\chi,\chi}$ by Lemma~\ref{mutation_lem} shown below.) 
Thus, we see that $\Hom_\calA(P_{\calL{\setminus}\chi},\delta_1)$ is surjective, and $(\chi\otimes_\kk K_\lambda)\otimes_\kk S\in P_{\calL{\setminus}\chi}$ by the condition $(\rmY 2)$. 
Since the functor $(-)^G$ gives an equivalence $\refl(G,S)\rightarrow\refl(R)$, we see that $\delta_1^G:(\chi\otimes_\kk K_\lambda\otimes_\kk S)^G\rightarrow(\chi\otimes_\kk S)^G=M_\chi$ is a right $\add_R(M_{\calL{\setminus}\chi})$-approximation of $M_\chi$. 
In addition, since the complex $C_{\lambda,\chi}$ is constructed from the Koszul resolution, 
there is no component of $(\chi\otimes_\kk K_\lambda)\otimes_\kk S$ that maps to zero via $\delta_1$, thus $\delta_1^G$ is miminal. 
Therefore, $M_{\calL{\setminus}\chi}\oplus\Ker\delta_1^G$ is the right mutation of $M_\calL$ at $M_\chi$. 
We remark that $\Ker\delta_1^G$ is not a module of covariants in general. 
Similarly, using the acyclicness of $C_{\calL{\setminus}\chi,\lambda,\chi}$ and the condition $(\rmY 2)$, we also have that $\delta_2^G$ is a right minimal 
$\add_R(M_{\calL{\setminus}\chi})$-approximation of $\Ker\delta_1^G$, 
and hence $M_{\calL{\setminus}\chi}\oplus\Ker\delta_2^G$ is the right mutation of $M_{\calL{\setminus}\chi}\oplus\Ker\delta_1^G$ at $\Ker\delta_1^G$. 
Repeating these processes, we finally arrive at 
$$M_{\calL{\setminus}\chi}\oplus\Ker\delta_{d_\lambda}^G=M_{\calL{\setminus}\chi}\oplus M_\nu,$$  
and this also gives an NCCR by Proposition~\ref{basic_prop_mutation}(2). 
Furthermore, this is splitting because $\dim_\kk(\chi\otimes_\kk\wedge^{d_\lambda}K_\lambda)=1$. 
The derived equivalence follows from Proposition~\ref{basic_prop_mutation}(3). 
\end{proof}

We denote $\mu_{\chi,\lambda}^+(M_\calL)\coloneqq M_{\calL{\setminus}\chi}\oplus M_\nu$, 
and call this the \emph{right mutation} of $M_\calL$ at $\chi$ with respect to $\lambda\in\rmY_{\calL\setminus\chi,\chi}$. 
We also define the \emph{left mutation} as $\mu_{\chi,\lambda}^-(M_\calL)=(\mu_{-\chi,-\lambda}^+(M_\calL^*))^*$. 

\begin{remark}
We remark that the mutations $\mu^\pm_{\chi,\lambda}$ are established by combining several mutations as shown in the proof of Theorem~\ref{mutation_NCCR}. 
In particular, if $d_\lambda>2$, then we have non-splitting NCCRs in the process of obtaining $\mu^\pm_{\chi,\lambda}(M_\calL)$ from $M_\calL$. 
This situation is quite different from the observation shown in \cite{Nak}, 
which discusses the mutations of splitting NCCRs for $3$-dimensional Gorenstein toric rings. 
\end{remark}

The next lemma asserts that these mutations do not depend on a choice of $\lambda\in\rmY_{\calL\setminus\chi,\chi}$. 

\begin{lemma}
\label{mutation_lem}
Let the notation be the same as above. 
For $\lambda,\lambda^\prime\in\rmY_{\calL\setminus\chi,\chi}$, we see that 
$\langle\lambda,\beta_{i_j}\rangle>0$ if and only if $\langle\lambda^\prime,\beta_{i_j}\rangle>0$. 
\end{lemma}

\begin{proof}
Suppose that $\beta_{i_j}$ satisfies $\langle\lambda,\beta_{i_j}\rangle>0$. 
Using the condition $(\rmY 2)$ for $\lambda$, we have that $\chi+\beta_{i_j}\in\calL{\setminus}\chi$. 
Since $\lambda^\prime$ separates $\chi$ from $\calL{\setminus}\chi$, we have that $\langle\lambda^\prime,\chi \rangle<\langle\lambda^\prime,\chi+\beta_{i_j} \rangle$, 
and hence $\langle\lambda^\prime,\beta_{i_j} \rangle>0$ holds. The converse is similar. 
\end{proof}

We note that $\mu_{\chi,\lambda}^+\neq\mu_{\chi,\lambda}^-$ in general. 
We also remark that even if $\lambda\in\rmY_{\calL\setminus\chi,\chi}$, there is a case where $-\lambda$ is not in $\rmY_{-(\calL\setminus\chi),-\chi}$
(the condition ($\rmY2$) does not hold in general), 
thus we can not define $\mu_{\chi,\lambda}^-(M_\calL)$ in such a situation. Also, we encounter the vice versa situation. 
If we can define both $\mu_{\chi,\lambda}^+$ and $\mu_{\chi,\lambda}^-$, then these are mutually inverse operations in the following sense. 

\begin{proposition}
\label{mutation_prop2}
Let the notation be the same as above. 
\begin{enumerate}[\rm (a)]
\setlength{\parskip}{0pt} 
\setlength{\itemsep}{3pt}
\item For a character $\chi\in\calL$, we assume that there exists a one parameter subgroup $\lambda_1\in Y_{\calL\setminus\chi,\chi}$. 
We consider the character $\nu$ of the form: 
$$\nu=\chi+\beta_{i_1}+\cdots+\beta_{i_{d}}$$
with $i_j\neq i_{j^\prime}$ if $j\neq j^\prime$, and $\langle\lambda_1,\beta_{i_j}\rangle>0$ for all $j=1,\cdots,d$ where $d\coloneqq d_{\lambda_1}$. 
If there exists a one parameter subgroup $-\lambda_2\in Y_{-(\calL\setminus\chi),-\nu}$, then we have 
$$\mu^-_{\nu,\lambda_2}(\mu_{\chi,\lambda_1}^+(M_\calL))=M_\calL.$$

\item For a character $\chi\in\calL$, we assume that there exists a one parameter subgroup $-\lambda_1\in Y_{-(\calL\setminus\chi),-\chi}$. 
We consider the character $\nu^\prime$ of the form: 
$$\nu^\prime=-\chi+\beta^\prime_{i_1}+\cdots+\beta^\prime_{i_{d^\prime}}$$
with $i_j\neq i_{j^\prime}$ if $j\neq j^\prime$, and $\langle -\lambda_1,\beta^\prime_{i_j}\rangle>0$ for all $j=1,\cdots,d^\prime$ where $d^\prime\coloneqq d_{-\lambda_1}$.
If there exists a one parameter subgroup $\lambda_2\in Y_{\calL\setminus\chi,\nu}$, then we have 
$$\mu^+_{\nu,\lambda_2}(\mu_{\chi,\lambda_1}^-(M_\calL))=M_\calL.$$
\end{enumerate}
\end{proposition}

\begin{proof}
(a) Let $B_\lambda\coloneqq \{\beta_i\mid\langle \lambda,\beta_i\rangle>0\}$, thus $B_{\lambda_1}=\{\beta_{i_1},\cdots,\beta_{i_d}\}$ by the definition. 
We first remark that $B_{\lambda_1}=B_{-\lambda_2}$ holds. 
In fact, for $\gamma_k=\sum_{j=1}^{d}\beta_{i_j}-\beta_{i_k}$ $(k=1,\cdots,d)$, we see that $\chi+\gamma_k\in\calL\setminus\{\chi\}$ 
by the condition $(\rmY 2)$. Thus, we have $\langle-\lambda_2,-\nu\rangle<\langle-\lambda_2, -\chi-\gamma_k\rangle$, 
and hence $\langle-\lambda_2,\beta_{i_k}\rangle>0$. Therefore, we have $B_{\lambda_1}\subset B_{-\lambda_2}$. 
We also have $B_{\lambda_1}\supset B_{-\lambda_2}$ by a similar argument. 
Then, we see that 
\begin{align*}
\mu^-_{\nu,\lambda_2}(\mu^+_{\chi,\lambda_1}(M_\calL))&=\left(\mu^+_{-\nu,-\lambda_2}\big(\big(\mu^+_{\chi,\lambda_1}(M_\calL)\big)^*\big)\right)^*=\left(\mu^+_{-\nu,-\lambda_2}(M_{\calL\setminus\chi}^*\oplus M_\nu^*)\right)^* \\
&=\big(\mu^+_{-\nu,-\lambda_2}(M_{-(\calL\setminus\chi)}\oplus M_{-\nu})\big)^*=(M_{-(\calL\setminus\chi)}\oplus M_{-\nu+\beta_{i_1}+\cdots+\beta_{i_d}})^* \\
&=(M_{-(\calL\setminus\chi)}\oplus M_{-\chi})^*=M_\calL. 
\end{align*}

(b) We similarly see that $B_{-\lambda_1}=B_{\lambda_2}$. 
Thus, we have 
\begin{align*}
\mu^+_{\nu,\lambda_2}(\mu^-_{\chi,\lambda_1}(M_\calL))&=\mu^+_{\nu,\lambda_2}\left(\big(\mu^+_{-\chi,-\lambda_1}(M_\calL^*)\big)^*\right)
=\mu^+_{\nu,\lambda_2}\left(\big(\mu^+_{-\chi,-\lambda_1}(M_{-\calL})\big)^*\right) \\
&=\mu^+_{\nu,\lambda_2}\left((M_{-(\calL\setminus\chi)}\oplus M_{\nu^\prime})^*\right)=\mu^+_{\nu,\lambda_2}(M_{\calL\setminus\chi}\oplus M_{-\nu^\prime}) \\
&=M_{\calL\setminus\chi}\oplus M_{-\nu^\prime+\beta^\prime_{i_1}+\cdots+\beta^\prime_{i_{d^\prime}}}=M_\calL. 
\end{align*}
\end{proof}

\begin{example}
Combining Theorem~\ref{NCCR_Segre} and Theorem~\ref{mutation_NCCR}, we have several NCCRs of Segre products of polynomial rings. 
For example, we consider the case of $r=2, t=3$. 
By Theorem~\ref{NCCR_Segre}, we see that $M_\calL=\bigoplus_{\eta\in\calL}M_\eta$ gives a splitting NCCR where $\calL=\{(0,0),(1,0),(0,1),(1,1)\}$. 
We pick $\chi=(1,0)$. Then, we see that $\lambda=(-1,1)\in\rmY_{\calL{\setminus}\chi,\chi}$. 
Since the weights $\beta_i$ satisfying $\langle \lambda,\beta_i\rangle>0$ are only $(0,1)$ (with the multiplicity two), 
we have that $\mu^+_{\chi,\lambda}(M_\calL)=M_{\calL{\setminus}\chi}\oplus M_\nu$ also gives a splitting NCCR by Theorem~\ref{mutation_NCCR}, 
where $\nu=(1,0)+(0,1)+(0,1)=(1,2)$. 

Repeating to apply the mutations $\mu^\pm_{\chi,\lambda}$, we have the \emph{exchange graph} shown in Figure~\ref{mutation_graph}.  
In this figure, each vertex is a splitting NCCR of $R$ given by the direct sum of divisorial ideals represented by red dots. 
Also, splitting NCCRs connected by an edge are transformed into each other using $\mu^+_{\chi,\lambda}$ and $\mu^-_{\chi,\lambda}$ 
for some $\chi\in\rmX(G)$ and $\lambda\in\rmY_{\calL{\setminus}\chi,\chi}. $
Since modules giving splitting NCCRs are infinite families, we only denote generators giving NCCRs. 
Here, a \emph{generator} is a module containing $M_\chi$ with $\chi=0$ as the direct summand. 

\begin{figure}[h!]
\begin{center}
{\scalebox{0.8}{
\begin{tikzpicture}

\node (A1) at (4,2.5)
{\scalebox{0.3}{
\begin{tikzpicture}
\coordinate (00) at (0,0); \coordinate (10) at (1,0); \coordinate (01) at (0,1); \coordinate (-10) at (-1,0); \coordinate (0-1) at (0,-1);  
\coordinate (11) at (1,1); \coordinate (-11) at (-1,1); \coordinate (-1-1) at (-1,-1); \coordinate (1-1) at (1,-1); 
\coordinate (12) at (1,2); \coordinate (21) at (2,1); \coordinate (-1-2) at (-1,-2); \coordinate (-2-1) at (-2,-1); 

\filldraw [color=lightgray] (1,0)--(1,1)--(0,1)--(-1,0)--(-1,-1)--(0,-1)--(1,0) ; 
\draw [step=1, gray] (-3,-3) grid (3,3);
\draw [color=gray] (1,0)--(1,1)--(0,1)--(-1,0)--(-1,-1)--(0,-1)--(1,0) ; 

\filldraw[red] (00) circle [radius=0.17] ; \filldraw[red] (10) circle [radius=0.17] ; \filldraw[red] (01) circle [radius=0.17] ; \filldraw (-10) circle [radius=0.17] ; 
\filldraw (0-1) circle [radius=0.17] ; \filldraw[red] (11) circle [radius=0.17] ; \filldraw (-11) circle [radius=0.17] ; \filldraw (-1-1) circle [radius=0.17] ;
\filldraw (1-1) circle [radius=0.17] ; \filldraw (21) circle [radius=0.17] ; \filldraw (12) circle [radius=0.17] ; 
\filldraw (-2-1) circle [radius=0.17] ; \filldraw (-1-2) circle [radius=0.17] ; 
\end{tikzpicture}
} }; 

\node (A2) at (3,-8.5)
{\scalebox{0.3}{
\begin{tikzpicture}
\filldraw [color=lightgray] (1,0)--(1,1)--(0,1)--(-1,0)--(-1,-1)--(0,-1)--(1,0) ; 
\draw [step=1, gray] (-3,-3) grid (3,3);
\draw [color=gray] (1,0)--(1,1)--(0,1)--(-1,0)--(-1,-1)--(0,-1)--(1,0) ; 

\filldraw[red] (00) circle [radius=0.17] ; \filldraw[red] (10) circle [radius=0.17] ; \filldraw (01) circle [radius=0.17] ; \filldraw (-10) circle [radius=0.17] ; 
\filldraw[red] (0-1) circle [radius=0.17] ; \filldraw (11) circle [radius=0.17] ; \filldraw (-11) circle [radius=0.17] ; \filldraw (-1-1) circle [radius=0.17] ;
\filldraw[red] (1-1) circle [radius=0.17] ; \filldraw (21) circle [radius=0.17] ; \filldraw (12) circle [radius=0.17] ; 
\filldraw (-2-1) circle [radius=0.17] ; \filldraw (-1-2) circle [radius=0.17] ; 
\end{tikzpicture}
} }; 

\node (A3) at (-3,8.5)
{\scalebox{0.3}{
\begin{tikzpicture}
\filldraw [color=lightgray] (1,0)--(1,1)--(0,1)--(-1,0)--(-1,-1)--(0,-1)--(1,0) ; 
\draw [step=1, gray] (-3,-3) grid (3,3);
\draw [color=gray] (1,0)--(1,1)--(0,1)--(-1,0)--(-1,-1)--(0,-1)--(1,0) ; 

\filldraw[red] (00) circle [radius=0.17] ; \filldraw (10) circle [radius=0.17] ; \filldraw[red] (01) circle [radius=0.17] ; \filldraw[red] (-10) circle [radius=0.17] ; 
\filldraw (0-1) circle [radius=0.17] ; \filldraw (11) circle [radius=0.17] ; \filldraw[red] (-11) circle [radius=0.17] ; \filldraw (-1-1) circle [radius=0.17] ;
\filldraw (1-1) circle [radius=0.17] ; \filldraw (21) circle [radius=0.17] ; \filldraw (12) circle [radius=0.17] ; 
\filldraw (-2-1) circle [radius=0.17] ; \filldraw (-1-2) circle [radius=0.17] ; 
\end{tikzpicture}
} }; 

\node (A4) at (-4,-2.5)
{\scalebox{0.3}{
\begin{tikzpicture}
\filldraw [color=lightgray] (1,0)--(1,1)--(0,1)--(-1,0)--(-1,-1)--(0,-1)--(1,0) ; 
\draw [step=1, gray] (-3,-3) grid (3,3);
\draw [color=gray] (1,0)--(1,1)--(0,1)--(-1,0)--(-1,-1)--(0,-1)--(1,0) ; 

\filldraw[red] (00) circle [radius=0.17] ; \filldraw (10) circle [radius=0.17] ; \filldraw (01) circle [radius=0.17] ; \filldraw[red] (-10) circle [radius=0.17] ; 
\filldraw[red] (0-1) circle [radius=0.17] ; \filldraw (11) circle [radius=0.17] ; \filldraw (-11) circle [radius=0.17] ; \filldraw[red] (-1-1) circle [radius=0.17] ;
\filldraw (1-1) circle [radius=0.17] ; \filldraw (21) circle [radius=0.17] ; \filldraw (12) circle [radius=0.17] ; 
\filldraw (-2-1) circle [radius=0.17] ; \filldraw (-1-2) circle [radius=0.17] ; 
\end{tikzpicture}
} }; 

\node (B1) at (8,0)
{\scalebox{0.3}{
\begin{tikzpicture}
\filldraw [color=lightgray] (1,0)--(1,1)--(0,1)--(-1,0)--(-1,-1)--(0,-1)--(1,0) ; 
\draw [step=1, gray] (-3,-3) grid (3,3);
\draw [color=gray] (1,0)--(1,1)--(0,1)--(-1,0)--(-1,-1)--(0,-1)--(1,0) ; 

\filldraw[red] (00) circle [radius=0.17] ; \filldraw[red] (10) circle [radius=0.17] ; \filldraw (01) circle [radius=0.17] ; \filldraw (-10) circle [radius=0.17] ; 
\filldraw (0-1) circle [radius=0.17] ; \filldraw[red] (11) circle [radius=0.17] ; \filldraw (-11) circle [radius=0.17] ; \filldraw (-1-1) circle [radius=0.17] ;
\filldraw (1-1) circle [radius=0.17] ; \filldraw[red] (21) circle [radius=0.17] ; \filldraw (12) circle [radius=0.17] ; 
\filldraw (-2-1) circle [radius=0.17] ; \filldraw (-1-2) circle [radius=0.17] ; 
\end{tikzpicture}
} }; 

\node (B2) at (0,5.5)
{\scalebox{0.3}{
\begin{tikzpicture}
\filldraw [color=lightgray] (1,0)--(1,1)--(0,1)--(-1,0)--(-1,-1)--(0,-1)--(1,0) ; 
\draw [step=1, gray] (-3,-3) grid (3,3);
\draw [color=gray] (1,0)--(1,1)--(0,1)--(-1,0)--(-1,-1)--(0,-1)--(1,0) ; 

\filldraw[red] (00) circle [radius=0.17] ; \filldraw (10) circle [radius=0.17] ; \filldraw [red](01) circle [radius=0.17] ; \filldraw[red] (-10) circle [radius=0.17] ; 
\filldraw (0-1) circle [radius=0.17] ; \filldraw[red] (11) circle [radius=0.17] ; \filldraw (-11) circle [radius=0.17] ; \filldraw (-1-1) circle [radius=0.17] ;
\filldraw (1-1) circle [radius=0.17] ; \filldraw (21) circle [radius=0.17] ; \filldraw (12) circle [radius=0.17] ; 
\filldraw (-2-1) circle [radius=0.17] ; \filldraw (-1-2) circle [radius=0.17] ; 
\end{tikzpicture}
} }; 

\node (B3) at (0,-5.5)
{\scalebox{0.3}{
\begin{tikzpicture}
\filldraw [color=lightgray] (1,0)--(1,1)--(0,1)--(-1,0)--(-1,-1)--(0,-1)--(1,0) ; 
\draw [step=1, gray] (-3,-3) grid (3,3);
\draw [color=gray] (1,0)--(1,1)--(0,1)--(-1,0)--(-1,-1)--(0,-1)--(1,0) ; 

\filldraw[red] (00) circle [radius=0.17] ; \filldraw[red] (10) circle [radius=0.17] ; \filldraw (01) circle [radius=0.17] ; \filldraw (-10) circle [radius=0.17] ; 
\filldraw[red] (0-1) circle [radius=0.17] ; \filldraw (11) circle [radius=0.17] ; \filldraw (-11) circle [radius=0.17] ; \filldraw[red] (-1-1) circle [radius=0.17] ;
\filldraw (1-1) circle [radius=0.17] ; \filldraw (21) circle [radius=0.17] ; \filldraw (12) circle [radius=0.17] ; 
\filldraw (-2-1) circle [radius=0.17] ; \filldraw (-1-2) circle [radius=0.17] ; 
\end{tikzpicture}
} }; 

\node (B4) at (-8,0)
{\scalebox{0.3}{
\begin{tikzpicture}
\filldraw [color=lightgray] (1,0)--(1,1)--(0,1)--(-1,0)--(-1,-1)--(0,-1)--(1,0) ; 
\draw [step=1, gray] (-3,-3) grid (3,3);
\draw [color=gray] (1,0)--(1,1)--(0,1)--(-1,0)--(-1,-1)--(0,-1)--(1,0) ; 

\filldraw[red] (00) circle [radius=0.17] ; \filldraw (10) circle [radius=0.17] ; \filldraw (01) circle [radius=0.17] ; \filldraw[red] (-10) circle [radius=0.17] ; 
\filldraw (0-1) circle [radius=0.17] ; \filldraw (11) circle [radius=0.17] ; \filldraw (-11) circle [radius=0.17] ; \filldraw[red] (-1-1) circle [radius=0.17] ;
\filldraw (1-1) circle [radius=0.17] ; \filldraw (21) circle [radius=0.17] ; \filldraw (12) circle [radius=0.17] ; 
\filldraw[red] (-2-1) circle [radius=0.17] ; \filldraw (-1-2) circle [radius=0.17] ; 
\end{tikzpicture}
} }; 

\node (C1) at (3,8.5)
{\scalebox{0.3}{
\begin{tikzpicture}
\filldraw [color=lightgray] (1,0)--(1,1)--(0,1)--(-1,0)--(-1,-1)--(0,-1)--(1,0) ; 
\draw [step=1, gray] (-3,-3) grid (3,3);
\draw [color=gray] (1,0)--(1,1)--(0,1)--(-1,0)--(-1,-1)--(0,-1)--(1,0) ; 

\filldraw[red] (00) circle [radius=0.17] ; \filldraw (10) circle [radius=0.17] ; \filldraw[red] (01) circle [radius=0.17] ; \filldraw (-10) circle [radius=0.17] ; 
\filldraw (0-1) circle [radius=0.17] ; \filldraw[red] (11) circle [radius=0.17] ; \filldraw (-11) circle [radius=0.17] ; \filldraw (-1-1) circle [radius=0.17] ;
\filldraw (1-1) circle [radius=0.17] ; \filldraw (21) circle [radius=0.17] ; \filldraw[red] (12) circle [radius=0.17] ; 
\filldraw (-2-1) circle [radius=0.17] ; \filldraw (-1-2) circle [radius=0.17] ; 
\end{tikzpicture}
} }; 

\node (C2) at (4,-2.5)
{\scalebox{0.3}{
\begin{tikzpicture}
\filldraw [color=lightgray] (1,0)--(1,1)--(0,1)--(-1,0)--(-1,-1)--(0,-1)--(1,0) ; 
\draw [step=1, gray] (-3,-3) grid (3,3);
\draw [color=gray] (1,0)--(1,1)--(0,1)--(-1,0)--(-1,-1)--(0,-1)--(1,0) ; 

\filldraw[red] (00) circle [radius=0.17] ; \filldraw[red] (10) circle [radius=0.17] ; \filldraw (01) circle [radius=0.17] ; \filldraw (-10) circle [radius=0.17] ; 
\filldraw[red] (0-1) circle [radius=0.17] ; \filldraw[red] (11) circle [radius=0.17] ; \filldraw (-11) circle [radius=0.17] ; \filldraw (-1-1) circle [radius=0.17] ;
\filldraw (1-1) circle [radius=0.17] ; \filldraw (21) circle [radius=0.17] ; \filldraw (12) circle [radius=0.17] ; 
\filldraw (-2-1) circle [radius=0.17] ; \filldraw (-1-2) circle [radius=0.17] ; 
\end{tikzpicture}
} }; 

\node (C3) at (-3,-8.5)
{\scalebox{0.3}{
\begin{tikzpicture}
\filldraw [color=lightgray] (1,0)--(1,1)--(0,1)--(-1,0)--(-1,-1)--(0,-1)--(1,0) ; 
\draw [step=1, gray] (-3,-3) grid (3,3);
\draw [color=gray] (1,0)--(1,1)--(0,1)--(-1,0)--(-1,-1)--(0,-1)--(1,0) ; 

\filldraw[red] (00) circle [radius=0.17] ; \filldraw (10) circle [radius=0.17] ; \filldraw (01) circle [radius=0.17] ; \filldraw (-10) circle [radius=0.17] ; 
\filldraw[red] (0-1) circle [radius=0.17] ; \filldraw (11) circle [radius=0.17] ; \filldraw (-11) circle [radius=0.17] ; \filldraw[red] (-1-1) circle [radius=0.17] ;
\filldraw (1-1) circle [radius=0.17] ; \filldraw (21) circle [radius=0.17] ; \filldraw (12) circle [radius=0.17] ; 
\filldraw (-2-1) circle [radius=0.17] ; \filldraw[red] (-1-2) circle [radius=0.17] ; 
\end{tikzpicture}
} }; 

\node (C4) at (-4,2.5)
{\scalebox{0.3}{
\begin{tikzpicture}
\filldraw [color=lightgray] (1,0)--(1,1)--(0,1)--(-1,0)--(-1,-1)--(0,-1)--(1,0) ; 
\draw [step=1, gray] (-3,-3) grid (3,3);
\draw [color=gray] (1,0)--(1,1)--(0,1)--(-1,0)--(-1,-1)--(0,-1)--(1,0) ; 

\filldraw[red] (00) circle [radius=0.17] ; \filldraw (10) circle [radius=0.17] ; \filldraw[red] (01) circle [radius=0.17] ; \filldraw[red] (-10) circle [radius=0.17] ; 
\filldraw (0-1) circle [radius=0.17] ; \filldraw (11) circle [radius=0.17] ; \filldraw (-11) circle [radius=0.17] ; \filldraw[red] (-1-1) circle [radius=0.17] ;
\filldraw (1-1) circle [radius=0.17] ; \filldraw (21) circle [radius=0.17] ; \filldraw (12) circle [radius=0.17] ; 
\filldraw (-2-1) circle [radius=0.17] ; \filldraw (-1-2) circle [radius=0.17] ; 
\end{tikzpicture}
} }; 

\node (D1) at (6,5.5)
{\scalebox{0.3}{
\begin{tikzpicture}
\filldraw [color=lightgray] (1,0)--(1,1)--(0,1)--(-1,0)--(-1,-1)--(0,-1)--(1,0) ; 
\draw [step=1, gray] (-3,-3) grid (3,3);
\draw [color=gray] (1,0)--(1,1)--(0,1)--(-1,0)--(-1,-1)--(0,-1)--(1,0) ; 

\filldraw[red] (00) circle [radius=0.17] ; \filldraw (10) circle [radius=0.17] ; \filldraw (01) circle [radius=0.17] ; \filldraw (-10) circle [radius=0.17] ; 
\filldraw (0-1) circle [radius=0.17] ; \filldraw[red] (11) circle [radius=0.17] ; \filldraw (-11) circle [radius=0.17] ; \filldraw (-1-1) circle [radius=0.17] ;
\filldraw (1-1) circle [radius=0.17] ; \filldraw[red] (21) circle [radius=0.17] ; \filldraw[red] (12) circle [radius=0.17] ; 
\filldraw (-2-1) circle [radius=0.17] ; \filldraw (-1-2) circle [radius=0.17] ; 
\end{tikzpicture}
} }; 

\node (D2) at (0,-2)
{\scalebox{0.3}{
\begin{tikzpicture}
\filldraw [color=lightgray] (1,0)--(1,1)--(0,1)--(-1,0)--(-1,-1)--(0,-1)--(1,0) ; 
\draw [step=1, gray] (-3,-3) grid (3,3);
\draw [color=gray] (1,0)--(1,1)--(0,1)--(-1,0)--(-1,-1)--(0,-1)--(1,0) ; 

\filldraw[red] (00) circle [radius=0.17] ; \filldraw[red] (10) circle [radius=0.17] ; \filldraw[red] (01) circle [radius=0.17] ; \filldraw (-10) circle [radius=0.17] ; 
\filldraw (0-1) circle [radius=0.17] ; \filldraw (11) circle [radius=0.17] ; \filldraw (-11) circle [radius=0.17] ; \filldraw[red] (-1-1) circle [radius=0.17] ;
\filldraw (1-1) circle [radius=0.17] ; \filldraw (21) circle [radius=0.17] ; \filldraw (12) circle [radius=0.17] ; 
\filldraw (-2-1) circle [radius=0.17] ; \filldraw (-1-2) circle [radius=0.17] ; 
\end{tikzpicture}
} }; 

\node (D3) at (0,-10.5)
{\scalebox{0.3}{
\begin{tikzpicture}
\filldraw [color=lightgray] (1,0)--(1,1)--(0,1)--(-1,0)--(-1,-1)--(0,-1)--(1,0) ; 
\draw [step=1, gray] (-3,-3) grid (3,3);
\draw [color=gray] (1,0)--(1,1)--(0,1)--(-1,0)--(-1,-1)--(0,-1)--(1,0) ; 

\filldraw[red] (00) circle [radius=0.17] ; \filldraw (10) circle [radius=0.17] ; \filldraw (01) circle [radius=0.17] ; \filldraw (-10) circle [radius=0.17] ; 
\filldraw[red] (0-1) circle [radius=0.17] ; \filldraw (11) circle [radius=0.17] ; \filldraw (-11) circle [radius=0.17] ; \filldraw (-1-1) circle [radius=0.17] ;
\filldraw[red] (1-1) circle [radius=0.17] ; \filldraw (21) circle [radius=0.17] ; \filldraw (12) circle [radius=0.17] ; 
\filldraw (-2-1) circle [radius=0.17] ; \filldraw[red] (-1-2) circle [radius=0.17] ; 
\end{tikzpicture}
} }; 

\node (D4) at (-6,5.5)
{\scalebox{0.3}{
\begin{tikzpicture}
\filldraw [color=lightgray] (1,0)--(1,1)--(0,1)--(-1,0)--(-1,-1)--(0,-1)--(1,0) ; 
\draw [step=1, gray] (-3,-3) grid (3,3);
\draw [color=gray] (1,0)--(1,1)--(0,1)--(-1,0)--(-1,-1)--(0,-1)--(1,0) ; 

\filldraw[red] (00) circle [radius=0.17] ; \filldraw (10) circle [radius=0.17] ; \filldraw (01) circle [radius=0.17] ; \filldraw[red] (-10) circle [radius=0.17] ; 
\filldraw (0-1) circle [radius=0.17] ; \filldraw (11) circle [radius=0.17] ; \filldraw[red] (-11) circle [radius=0.17] ; \filldraw (-1-1) circle [radius=0.17] ;
\filldraw (1-1) circle [radius=0.17] ; \filldraw (21) circle [radius=0.17] ; \filldraw (12) circle [radius=0.17] ; 
\filldraw[red] (-2-1) circle [radius=0.17] ; \filldraw (-1-2) circle [radius=0.17] ; 
\end{tikzpicture}
} }; 

\node (D1*) at (-6,-5.5)
{\scalebox{0.3}{
\begin{tikzpicture}
\filldraw [color=lightgray] (1,0)--(1,1)--(0,1)--(-1,0)--(-1,-1)--(0,-1)--(1,0) ; 
\draw [step=1, gray] (-3,-3) grid (3,3);
\draw [color=gray] (1,0)--(1,1)--(0,1)--(-1,0)--(-1,-1)--(0,-1)--(1,0) ; 

\filldraw[red] (00) circle [radius=0.17] ; \filldraw (10) circle [radius=0.17] ; \filldraw (01) circle [radius=0.17] ; \filldraw (-10) circle [radius=0.17] ; 
\filldraw (0-1) circle [radius=0.17] ; \filldraw (11) circle [radius=0.17] ; \filldraw (-11) circle [radius=0.17] ; \filldraw[red] (-1-1) circle [radius=0.17] ;
\filldraw (1-1) circle [radius=0.17] ; \filldraw (21) circle [radius=0.17] ; \filldraw (12) circle [radius=0.17] ; 
\filldraw[red] (-2-1) circle [radius=0.17] ; \filldraw[red] (-1-2) circle [radius=0.17] ; 
\end{tikzpicture}
} }; 

\node (D2*) at (0,2)
{\scalebox{0.3}{
\begin{tikzpicture}
\filldraw [color=lightgray] (1,0)--(1,1)--(0,1)--(-1,0)--(-1,-1)--(0,-1)--(1,0) ; 
\draw [step=1, gray] (-3,-3) grid (3,3);
\draw [color=gray] (1,0)--(1,1)--(0,1)--(-1,0)--(-1,-1)--(0,-1)--(1,0) ; 

\filldraw[red] (00) circle [radius=0.17] ; \filldraw (10) circle [radius=0.17] ; \filldraw (01) circle [radius=0.17] ; \filldraw[red] (-10) circle [radius=0.17] ; 
\filldraw[red] (0-1) circle [radius=0.17] ; \filldraw[red] (11) circle [radius=0.17] ; \filldraw (-11) circle [radius=0.17] ; \filldraw (-1-1) circle [radius=0.17] ;
\filldraw (1-1) circle [radius=0.17] ; \filldraw (21) circle [radius=0.17] ; \filldraw (12) circle [radius=0.17] ; 
\filldraw (-2-1) circle [radius=0.17] ; \filldraw (-1-2) circle [radius=0.17] ; 
\end{tikzpicture}
} }; 

\node (D3*) at (0,10.5)
{\scalebox{0.3}{
\begin{tikzpicture}
\filldraw [color=lightgray] (1,0)--(1,1)--(0,1)--(-1,0)--(-1,-1)--(0,-1)--(1,0) ; 
\draw [step=1, gray] (-3,-3) grid (3,3);
\draw [color=gray] (1,0)--(1,1)--(0,1)--(-1,0)--(-1,-1)--(0,-1)--(1,0) ; 

\filldraw[red] (00) circle [radius=0.17] ; \filldraw (10) circle [radius=0.17] ; \filldraw[red] (01) circle [radius=0.17] ; \filldraw (-10) circle [radius=0.17] ; 
\filldraw (0-1) circle [radius=0.17] ; \filldraw (11) circle [radius=0.17] ; \filldraw[red] (-11) circle [radius=0.17] ; \filldraw (-1-1) circle [radius=0.17] ;
\filldraw (1-1) circle [radius=0.17] ; \filldraw (21) circle [radius=0.17] ; \filldraw[red] (12) circle [radius=0.17] ; 
\filldraw (-2-1) circle [radius=0.17] ; \filldraw (-1-2) circle [radius=0.17] ; 
\end{tikzpicture}
} }; 

\node (D4*) at (6,-5.5)
{\scalebox{0.3}{
\begin{tikzpicture}
\filldraw [color=lightgray] (1,0)--(1,1)--(0,1)--(-1,0)--(-1,-1)--(0,-1)--(1,0) ; 
\draw [step=1, gray] (-3,-3) grid (3,3);
\draw [color=gray] (1,0)--(1,1)--(0,1)--(-1,0)--(-1,-1)--(0,-1)--(1,0) ; 

\filldraw[red] (00) circle [radius=0.17] ; \filldraw[red] (10) circle [radius=0.17] ; \filldraw (01) circle [radius=0.17] ; \filldraw (-10) circle [radius=0.17] ; 
\filldraw (0-1) circle [radius=0.17] ; \filldraw (11) circle [radius=0.17] ; \filldraw (-11) circle [radius=0.17] ; \filldraw (-1-1) circle [radius=0.17] ;
\filldraw[red] (1-1) circle [radius=0.17] ; \filldraw[red] (21) circle [radius=0.17] ; \filldraw (12) circle [radius=0.17] ; 
\filldraw (-2-1) circle [radius=0.17] ; \filldraw (-1-2) circle [radius=0.17] ; 
\end{tikzpicture}
} }; 

\draw[line width=0.02cm] (A1)--(B1); 
\draw[line width=0.02cm] (A1)--(B2); 
\draw[line width=0.02cm] (A1)--(C1); 
\draw[line width=0.02cm] (A1)--(C2); 
\draw[line width=0.02cm] (A1)--(D2); 
\draw[line width=0.02cm] (B1)--(C2); 
\draw[line width=0.02cm] (B1)--(D1); 
\draw[line width=0.02cm] (B1)--(D4*); 
\draw[line width=0.02cm] (B2)--(C1); 
\draw[line width=0.02cm] (C1)--(D1); 
\draw[line width=0.02cm] (A2)--(C2); 
\draw[line width=0.02cm] (B3)--(C2); 
\draw[line width=0.02cm] (C2)--(D2*); 
\draw[line width=0.02cm] (B2)--(D2*); 
\draw[line width=0.02cm] (B2)--(C4); 
\draw[line width=0.02cm] (A3)--(B2); 
\draw[line width=0.02cm] (B3)--(D2); 
\draw[line width=0.02cm] (C4)--(D2); 
\draw[line width=0.02cm] (A4)--(D2*); 
\draw[line width=0.02cm] (A2)--(D4*); 
\draw[line width=0.02cm] (B3)--(C3); 
\draw[line width=0.02cm] (A4)--(C3); 
\draw[line width=0.02cm] (C3)--(D3); 
\draw[line width=0.02cm] (C3)--(D1*); 
\draw[line width=0.02cm] (A2)--(B3); 
\draw[line width=0.02cm] (A4)--(B3); 
\draw[line width=0.02cm] (A2)--(D3); 
\draw[line width=0.02cm] (A3)--(C4); 
\draw[line width=0.02cm] (A4)--(C4); 
\draw[line width=0.02cm] (B4)--(C4); 
\draw[line width=0.02cm] (A3)--(D4); 
\draw[line width=0.02cm] (A3)--(D3*); 
\draw[line width=0.02cm] (C1)--(D3*); 
\draw[line width=0.02cm] (A4)--(B4); 
\draw[line width=0.02cm] (B4)--(D1*); 
\draw[line width=0.02cm] (B4)--(D4); 
\end{tikzpicture}
} }
\end{center}
\caption{The exchange graph of generators giving splitting NCCRs for the Segre product of polynomial rings with $r=2, t=3$}
\label{mutation_graph}
\end{figure}
\end{example}

\subsection*{Acknowledgement} 
The authors thank \v{S}pela \v{S}penko for valuable discussions concerning NCCRs and for explaining results in \cite{SpVdB}. 
The authors also thank the anonymous referee for valuable comments. 

The first author is partially supported by JSPS Grant-in-Aid for Young Scientists (B) 17K14177. 
The second author is supported by World Premier International Research Center Initiative (WPI initiative), MEXT, Japan, 
and JSPS Grant-in-Aid for Young Scientists (B) 17K14159. 


\end{document}